\documentclass[english]{amsproc}
\usepackage[T2A]{fontenc}
\usepackage[english]{babel}
\usepackage{graphics}
\usepackage{amsfonts, amssymb, amscd}
\usepackage[dvips]{graphicx}
\usepackage{latexsym}

\newcommand{\ve}[1]{\bar{#1}}

\DeclareMathOperator{\link}{link}
 \DeclareMathOperator{\pt}{pt}

 \DeclareMathOperator{\Ob}{Ob}
\DeclareMathOperator{\colim}{colim}
\DeclareMathOperator{\hocolim}{hocolim}
\DeclareMathOperator{\Tor}{Tor} \DeclareMathOperator{\rank}{rank}
\DeclareMathOperator{\Hilb}{Hilb}

\newcommand{\ST}[1]{\mbox{\upshape\footnotesize #1}}

\newcommand{\Zo}{\mathbb{Z}}
\newcommand{\Ro}{\mathbb{R}}

\newcommand{\ko}{\Bbbk}
\newcommand{\Rg}{\Ro_{\geqslant}}

\newcommand{\bigast}{\divideontimes}
\newcommand{\Z}{\mathcal{Z}}
\newcommand{\F}{\mathcal{F}}
\newcommand{\zer}{\hat{\sigma}}
\newcommand{\Ks}{\underline{K_{\alpha}}}
\newcommand{\Ps}{\underline{P_{\alpha}}}

\newcommand{\Hr}{\tilde{H}}
\newcommand{\bet}{b}
\newcommand{\bett}{\tilde{b}}
\newcommand{\betar}{\tilde{\beta}}
\newcommand{\ts}{\bar{t}}

\newcommand{\cat}{\ST{CAT}}
\newcommand{\Top}{\ST{TOP}}

\newcounter{stmcounter}[section]

\newcounter{defcounter}[section]

\numberwithin{equation}{section}

\renewcommand{\thestmcounter}{\thesection.\arabic{stmcounter}}

\newcommand{\ex}{\par\vspace{0.5 cm}\noindent\refstepcounter{stmcounter}\textsc{Example \thestmcounter.}\quad}
\newcommand{\rem}{\par\vspace{0.5 cm}\noindent\refstepcounter{stmcounter}\textsc{Remark \thestmcounter.}\quad}

\newtheorem{obser}[stmcounter]{Observation}
\newtheorem{cor}[stmcounter]{Corollary}
\newtheorem{stm}[stmcounter]{Statement}
\newtheorem{thm}[stmcounter]{Theorem}
\newtheorem{prop}[stmcounter]{Proposition}
\newtheorem{lemma}[stmcounter]{Lemma}
\newtheorem{defin}[stmcounter]{Definition}

\begin{document}

\title{Composition of simplicial complexes, polytopes and multigraded Betti numbers\quad}
\author{Ayzenberg Anton}
\email{ayzenberga@gmail.com}

\begin{abstract}
For a simplicial complex $K$ on $m$ vertices and simplicial
complexes $K_1,\ldots,K_m$ a composed simplicial complex
$K(K_1,\ldots,K_m)$ is introduced. This construction generalizes
an iterated simplicial wedge construction studied by A. Bahri, M.
Bendersky, F. R. Cohen and S. Gitler and allows to describe the
combinatorics of generalized joins of polytopes
$P(P_1,\ldots,P_m)$ defined by G. Agnarsson in most important
cases. The composition defines a structure of an operad on a set
of finite simplicial complexes, in which a complex on $m$ vertices
is viewed as an $m$-adic operation. We prove the following: (1) a
composed complex $K(K_1,\ldots,K_m)$ is a simplicial sphere iff
$K$ is a simplicial sphere and $K_i$ are the boundaries of
simplices; (2) a class of spherical nerve-complexes is closed
under the operation of composition (3) finally, we express
multigraded Betti numbers of $K(K_1,\ldots,K_m)$ in terms of
multigraded Betti numbers of $K, K_1,\ldots,K_m$ using a
composition of generating functions.
\end{abstract}

\maketitle

\section{Introduction}\label{sectIntro}

In toric topology multiple connections between convex polytopes,
simplicial complexes, topological spaces and Stanley--Reisner
algebras are studied. Starting with a simple polytope $P$ one
constructs a moment-angle manifold $\Z_P$ with a torus action such
that its orbit space is the polytope $P$ itself. On the other
hand, a simplicial complex $\partial P^*$ gives rise to a
moment-angle complex $\Z_{\partial P^*}(D^2,S^1)$. This complex is
homeomorphic to $\Z_P$ and possesses a natural cellular structure
which allows to describe its cohomology ring:
$H^*(\Z_P;\ko)\cong\Tor^{*,*}_{\ko[m]}(\ko[\partial P^*],\ko)$.
This consideration can be used to translate topological problems
to the language of Stanley--Reisner algebras and vice-versa.
Moreover, the cohomology ring $H^*(\Z_P;\ko)$ carries an
information about the combinatorics of the polytope $P$ from which
we started.

With some modifications this setting can be generalized to
nonsimple polytopes. If $P$ is a convex polytope (possibly
nonsimple), then the moment-angle space $\Z_P$ is defined as an
intersection of real quadrics (but in nonsimple case $\Z_P$ is not
a manifold). A simplicial complex $K_P$, called the nerve-complex
\cite{AB}, is associated to each polytope (in nonsimple case $K_P$
is not a simplicial sphere). The complex $K_P$ carries a complete
information on the combinatorics of $P$ and its properties are
similar to simplicial spheres. Generally there is a homotopy
equivalence $\Z_P\simeq \Z_{K_P}(D^2,S^1)$. An open question is to
describe the properties of Stanley--Reisner algebras $\ko[K_P]$
and cohomology rings
$H^*(\Z_P;\ko)\cong\Tor^{*,*}_{\ko[m]}(\ko[K_P],\ko)$ for
nonsimple convex polytopes.

In the work of A.\,Bahri, M.\,Bendersky, F.\,R.\,Cohen and
S.\,Gitler \cite{BBCGit} a new construction is described, which
allows to build a simple polytope $P(l_1,\ldots,l_m)$ from a given
simple polytope $P$ with $m$ facets and an array
$(l_1,\ldots,l_m)$ of natural numbers. A simplicial complex
$\partial P(l_1,\ldots,l_m)^*$ can be described combinatorially in
terms of missing faces. Such description gives a representation of
$\Z_{\partial P(l_1,\ldots,l_m)^*}(D^2,S^1)$ as a polyhedral
product $\Z_{P^*}(\underline{(D^{2l_i},S^{2l_i-1})})$ which leads,
in particular, to alternative representation of the cohomology
ring $H^*(\Z_{P(l_1,\ldots,l_m)})$.

The idea of treating nonsimple polytopes can be used to capture a
wider class of examples and find more general constructions. One
of the constructions is known in convex geometry (we refer to the
work of Geir Agnarsson \cite{Agn}). Given a polytope $P\subset
\Rg^m$ and polytopes $P_1,\ldots,P_m$ a new polytope
$P(P_1,\ldots,P_m)$ is constructed. This polytope generally
depends on geometrical representation of $P\subset \Rg^m$, but
under some restrictions the construction can be made
combinatorial. In particular cases this construction gives the
iterated polytope $P(l_1,\ldots,l_m)$ from the work \cite{BBCGit}.
Note, that the polytope $P(P_1,\ldots,P_m)$ may be nonsimple even
in the case when all the polytopes $P,P_1,\ldots,P_m$ are simple.

In this work we introduce a new operation on the set of abstract
simplicial complexes $K,K_1,\ldots,K_m \mapsto K(K_1,\ldots,K_m)$.
This operation corresponds to the operation $P(P_1,\ldots,P_m)$ on
convex polytopes and generalizes the constructions of
\cite{BBCGit}. The work is organized as follows:
\begin{enumerate}
\item We review the construction of $K_P$ and the definition of
abstract spherical nerve-complex from the work \cite{AB}. Section
\ref{sectPrelim}.

\item The construction of $P(P_1,\ldots,P_m)$. We give a few
equivalent descriptions of this polytope and specialize the
conditions under which $P(P_1,\ldots,P_m)$ is well defined on
combinatorial polytopes. Section \ref{sectCompPolytopes}.

\item Given a simplicial complex $K$ on $m$ vertices and
simplicial complexes $K_1,\ldots,K_m$ we define a composed
simplicial complex $K(K_1,\ldots,K_m)$, which is a central object
of the work. Two equivalent definitions are provided: one is
combinatorial, another describes $K(K_1,\ldots,K_m)$ as an
analogue of polyhedral product called polyhedral join. It is shown
that $K(\partial\Delta_{[l_1]},\ldots,\partial\Delta_{[l_m]}) =
K(l_1,\ldots,l_m)$ --- an iterated simplicial wedge construction
from the work \cite{BBCGit}. We prove that $K_{P(P_1,\ldots,P_m)}
= K_P(K_{P_1},\ldots,K_{P_m})$. Section \ref{sectCompSimpComp}.

\item Polyhedral products defined by composed simplicial
complexes. In section \ref{sectPolyProducts} we review and
generalize some results from \cite{BBCGit}.

\item In section \ref{sectCompHomotopySpheres} the
structure of composed simplicial complexes is studied. At first we
describe the homotopy type of $K(K_1,\ldots,K_m)$. It happens that
$K(K_1,\ldots,K_m)\simeq K\ast K_1\ast\ldots\ast K_m$. The
problem: for which choice of $K,K_1,\ldots,K_m$ the complex
$K(K_1,\ldots,K_m)$ is a sphere? The answer: only in the case,
when $K$ is a sphere and $K_i =
\partial\Delta_{[l_i]}$. Thus the class of simplicial spheres is
not closed under the composition. Nevertheless, if
$K,K_1,\ldots,K_m$ are spherical nerve-complexes, then so is
$K(K_1,\ldots,K_m)$.

\item In section \ref{sectMultBettiNums} we describe the
multigraded Betti numbers of $K(K_1,\ldots,K_m)$. There is a
simple formula which expresses these numbers in terms of
multigraded Betti numbers of $K, K_1,\ldots,K_m$. Applying this
formula to $\partial\Delta_{[2]}(K_1,K_2)$ and $o^2(K_1,K_2)$,
where $o^2$ is the complex with $2$ ghost vertices, gives the
result of \cite{AB}. Using the connection between bigraded Betti
numbers and $h$-polynomial, found by V.M.Buch\-sta\-ber and
T.E.Panov \cite{BP}, in section \ref{sectEnumPolynomials} we
provide formulas for $h$-polynomials of compositions in some
particular cases. Some of these formulas were found earlier by
Yu.Ustonovsky \cite{Ust}.

\end{enumerate}

The following notation and conventions are used. The simplicial
complex $K$ on a set of vertices $[m]$ is the system of subsets
$K\subseteq 2^{[m]}$, such that $I\in K$ and $J\subset I$ implies
$J\in K$. A vertex $i\in[m]$ such that $\{i\}\notin K$ is called
ghost vertex. If $I\in K$, then $\link_KI$ is the simplicial
complex on a set $[m]\setminus I$ such that $J\in \link_KI
\Leftrightarrow J\sqcup I\in K$. Note that a link may have ghost
vertices even if $K$ does not have them. From the geometrical
point of view the complex does not change when ghost vertices are
omitted. We use the same symbol for the simplicial complex $K$ and
its geometrical realization. The complex $K$ is called a
simplicial sphere if it is PL-homeomorphic to the boundary of a
simplex (we omit ghost vertices if necessary). Simplicial complex
$K$ is called a generalized homological sphere (or Gorenstein*
complex) if $K$ and all its links have homology of spheres of
corresponding dimensions. If $K$ is a simplicial sphere (resp.
Gorenstein* complex) then so is $\link_KI$ for each $I\in K$.

If $A\subset [m]$, then full subcomplex $K_A$ is the complex on
$A$ such that $J\in K_A \Leftrightarrow J\in K$. We denote the
full simplex on the set $[m]$ by $\Delta_{[m]}$, it has dimension
$m-1$. Its boundary $\partial \Delta_{[m]}$ --- is complex on
$[m]$, consisting of all proper subsets of $[m]$.

The notation $\ve{x} = (x_1,\ldots,x_m)\in \Ro^m$ is used for
arrays of numbers, and $\langle \ve{x},\ve{y}\rangle$ denotes the
sum $x_1y_1+x_2y_2+\ldots+x_my_m$. Sometimes double arrays will be
used: $\ve{x} = (\ve{x}_1,\ldots,\ve{x}_m) =
(x_{11},\ldots,x_{1l_1},\ldots,x_{m1},\ldots,x_{ml_m})$.

I wish to thank Anthony Bahri for the private discussion in which
he explained the geometrical meaning of the simplicial wedge
construction and for his comments on the subject of this work. I
am also grateful to Nickolai Erokhovets for paying my attention to
the work of Geir Agnarsson \cite{Agn}.

\section{Polytopes and nerve-complexes}\label{sectPrelim}

Let $P$ be an $n$-dimensional polytope and let
$\{\F_1,\ldots,\F_m\}$ be the set of all its facets. Consider a
simplicial complex $K_P$ on the set $[m]=\{1,\ldots,m\}$ called
the nerve-complex of a polytope $P$, defined by the condition
$I=\{i_1,\ldots,i_k\}\in K_P$ whenever
$\F_{i_1}\cap\ldots\cap\F_{i_k}\neq\varnothing$. The complex $K_P$
is thus the nerve of the closed covering of the boundary $\partial
P$ by facets.

\ex\label{examplPsimpleKPsphere} If $P$ is simple, then $K_P$
coincides with a boundary of a dual simplicial polytope: $K_P =
\partial P^*$. In this case $K_P$ is a simplicial sphere. It can
be shown that $K_P$ is not a sphere if $P$ is not simple.
\\
\\
As shown in \cite{AB} nerve-complexes are nice substitutes for
nonsimple polytopes. In particular, the moment-angle space $\Z_P$
of any convex polytope $P$ is homotopy equivalent to the
moment-angle complex $\Z_{K_P}(D^2,S^1)$, the Buchstaber numbers
$s(P)$ and $s(K_P)$ are equal, etc.

There are necessary conditions on the complex $K$ to be the
nerve-complex of some convex polytope. These conditions are
gathered in the notion of a \textbf{spherical nerve-complex}.

Let $K$ be a simplicial complex, $M(K)$ --- the set of its maximal
(under inclusion) simplices. Let $F(K) = \{I\in K\mid I=\cap J_i$,
where $J_i\in M(K)\}$. The set $F(K)$ is partially ordered by
inclusion. It can be shown (see \cite{AB}) that for each simplex
$I\notin F(K)$ the complex $\link_KI$ is contractible.

\begin{defin}[Spherical nerve-complex]
Simplicial complex $K$ is called a spherical nerve-complex of rank
$n$ if the following conditions hold:

\begin{itemize}
\item $\varnothing\in F(K)$, i.e. intersection of all maximal
simplices of $K$ is empty;

\item $F(K)$ is a graded poset of rank $n$ (it means that all its
saturated chains have the cardinality $n+1$). In this case the
rank function $\rank\colon F(K)\to \Zo_{\geqslant}$ is defined,
such that $\rank(I) =$ the cardinality of saturated chain from
$\varnothing$ to $I$ minus $1$.

\item For any simplex $I\in F(K)$ the simplicial complex
$\link_K I$ is homotopy equivalent to a sphere $S^{n-\rank(I)-1}$.
Here, by definition, $\link_K\varnothing=K$ and
$S^{-1}=\varnothing$.
\end{itemize}
\end{defin}

\begin{stm}
If $P$ is an $n$-dimensional polytope, then $K_P$ is a spherical
nerve-complex of rank $n$ and, moreover, the poset $F(K_P)$ is
isomorphic to the poset of faces of $P$ ordered by reverse
inclusion.
\end{stm}

As a corollary, the poset of faces of $P$ can be restored from
$K_P$, thus $K_P$ is a complete invariant of a combinatorial
polytope $P$.

\section{Composition of polytopes}\label{sectCompPolytopes}


Let $[m]=\{1,\ldots,m\}$ be a finite set and $\triangle_{[m]}$ be
a standard $(m-1)$-dimensional simplex in $\Ro^m$ given by
$\{\ve{x}=(x_1,\ldots,x_m)\in\Ro^m\mid x_i\geqslant 0; \sum
x_i=1\}$. The convex polytope $P\subset\Ro^m$ will be called
\textbf{stochastic} if $P\subseteq \triangle_{[m]}$. The following
definition is due to \cite[def.4.5]{Agn}.

\begin{defin}\label{definComposedPolytope}
Let $P\subseteq \Ro^m$ and $P_i\subseteq \Ro^{l_i}$ for $i\in[m]$
be stochastic polytopes. The polytope
\begin{multline}
P(P_1,\ldots,P_m)=\{(t_1\ve{x}_1,t_2\ve{x}_2,\ldots,t_m\ve{x}_m)\in
\Ro^{\sum l_i}\mid \\\mid \ve{t}=(t_1,\ldots,t_m)\in P,
\ve{x}_i\in P_i \mbox{ for each } i\}
\end{multline}
is called the composition of polytopes $P$ and $\{P_i\}$.
\end{defin}

In \cite{Agn} this operation is called the action of $P$.

\ex\label{examplDeltaPolyIsJoin} $\triangle_{[m]}(P_1,\ldots,P_m)
= P_1\ast\ldots\ast P_m$
--- the join of polytopes.
\\
\\
The original motivation of definition \ref{definComposedPolytope}
was to extend the notion of the join to more general convex sets
of parameters $t_i$.

\rem Definition \ref{definComposedPolytope} depends crucially on
the geometrical representation of polytopes, not only their
combinatorial type.

\begin{defin}\label{definNatStochPoly}
Let $L\subseteq \Ro^m$ be an affine $n$-dimensional subspace such
that $P = L\cap \Rg^m$ is a nonempty bounded set (thus a
polytope). If $P$ is a stochastic polytope and every facet
$\F_i\subset P$ is defined uniquely as $\F_i=P\cap \{x_i=0\}$ we
call $P$ a \textbf{natural (stochastic) polytope}.
\end{defin}

\rem A natural stochastic polytope $P$ in $\Ro^m$ has exactly $m$
facets.
\\
\\
For a point $\ve{x}\in \Rg^m$ define $\zer(\ve{x}) = \{i\in[m]\mid
x_i=0\}$.

\rem\label{remarkNerveStochasticDescr} For a natural stochastic
polytope $P\subseteq \Rg^m$ the nerve-complex can be defined by
the condition: $I\in K_P$, whenever there exists a point
$\ve{x}\in P$ such that $I\subseteq \zer(\ve{x})$. Indeed, $I\in
K_P$ implies that $\bigcap_{i\in I}\F_i\neq\varnothing$. Let
$\ve{x}\in \bigcap_{i\in I}\F_i$. Then $x_i=0$ for each $i\in I$
therefore $I\in\zer(\ve{x})$.

\begin{obser}
Any polytope $P$ is affine equivalent to a natural stochastic
polytope.
\end{obser}

\begin{obser}\label{obserNonnegCoeffs}
The space $L\subseteq \Ro^m$ in the definition
\ref{definNatStochPoly} can be defined by the system of affine
relations $L = \{\ve{x}\in \Ro^m\mid \sum_j c_i^jx_j+d_i=0$ for
$i=1,\ldots, m-n\}$ where all the coefficients $c_i^j$ are
positive and $d_i=-1$.
\end{obser}

\begin{proof}[Proof of both observations]
Let
\begin{equation}\label{equatPolyHalfspaces}
P = \{\ve{y}\in \Ro^n\mid \langle \ve{a}_i,
\ve{y}\rangle+b_i\geqslant 0, i\in [m]\} \end{equation} be a
representation of $P$ as an intersection of halfspaces, where
$\ve{a}_i$ is the inner normal vector to the $i$-th facet (we
suppose that there are no excess inequalities in
\eqref{equatPolyHalfspaces} and $|\ve{a}_i|=1$).

Consider an affine embedding $j_P\colon \Ro^n\to \Ro^m$, given by
$j_P(\ve{y}) = (\langle \ve{a}_1, \ve{y}\rangle+b_1,\ldots,
\langle \ve{a}_m, \ve{y}\rangle+b_m)$. Obviously, $j_P(P)\subseteq
\Rg^m$ and, moreover, $j_P(P)=j_P(\Ro^n)\cap\Rg^m$. Denote the
affine subspace $j_P(\Ro^n)$ by $L$. This subspace is given by the
system of affine relations $L = \{\ve{x}\in \Ro^m\mid
\langle\ve{c}_i,\ve{x}\rangle+d_i=0$ for $i=1,\ldots, m-n\}$. The
facets of $j_P(P)$ are given by $j_P(P)\cap\{x_i=0\}$.

Notice that there is a relation $\sum S_i\ve{a}_i=0$ by Minkowski
theorem, where $S_i>0$ are the $(n-1)$-volumes of facets. Then one
of the affine relations for $L$ has the form $\sum_i S_ix_i+d = 0$
with all the coefficients $S_i$ strictly positive. Adding this
relation multiplied by large enough number to other relations
leads to a system of relations with positive coefficients.

Now divide each relation by $d_i$ to get the relations of the form
$\sum c_i^jx_j=1$. Set new variables $x_j'=c_1^jx_j$ to transform
one of the relations to the form $\sum x_j=1$. This gives a
stochastic polytope in $\Ro^m$. Observations proved.
\end{proof}

\begin{prop}\label{propDescrCompPolyInRelations}
Let $P\in \Ro^m$ be a natural stochastic polytope given by
$P=\Rg^m\cap \{\langle \ve{c}_i,\ve{x}\rangle = 1, i=1,\ldots,
m-n\}$, $\ve{c}_i = (c_i^{1},\ldots, c_i^{m})$ and for each $i\in
[m]$ a natural stochastic polytope $P_i\in \Ro^{l_i}$ is given by
$P_i=\Rg^{l_i}\cap \{\langle \ve{c}_{ij_i},\ve{x_i}\rangle = 1,
j_i=1,\ldots, l_i-n_i\}$, $\ve{c}_{ij_i} = (c_{ij_i}^{1},\ldots,
c_{ij_i}^{l_i})$. Then the polytope $P(P_1,\ldots,P_m)$ is a
natural stochastic polytope described by the system
\begin{multline}\label{equatComposedRelation}
P(P_1,\ldots,P_m) = \{(\ve{x}_1,\ldots,\ve{x}_m)\in
\Rg^{l_1}\times\ldots\times\Rg^{l_m}=\Rg^{\sum l_i}\mid\\\mid
c_i^{1}\langle \ve{c}_{1j_1},\ve{x}_1\rangle+c_i^{2}\langle
\ve{c}_{2j_2},\ve{x}_2\rangle+\ldots+c_i^{m}\langle
\ve{c}_{mj_m},\ve{x}_m\rangle = 1\}
\end{multline}
\end{prop}

\begin{proof}
By direct substitution $P(P_1,\ldots,P_m)$ as defined in
\ref{definComposedPolytope} satisfies all the specified affine
relations. On the contrary let $\ve{x} =
(\ve{x}_1,\ldots,\ve{x}_m)\in \Rg^{\sum l_i}$ satisfies relations
\eqref{equatComposedRelation} for all $i,j_1,\ldots,j_m$. Denote
$\langle \ve{c}_{ij_i},\ve{x}_i\rangle\in \Ro$ by $t_i(j_i)$. Then
$t_i(j_i)\geqslant 0$ (by nonnegativity of coefficients in affine
relations) and
$c_i^1t_1(j_1)+c_i^2t_2(j_2)+\ldots+c_i^mt_m(j_m)=1$ for each $i$,
therefore $\ve{t}(\ve{j}) = (t_1(j_1),\ldots,t_m(j_m))\in P$.

Let us show that $t_i(j_i)$ does not actually depend on $j_i$.
Consider first entry $j_1$ for simplicity. Let $j_1$ and $j_1'$ be
different indices. The point $\ve{x}$ satisfies the relations
$$
c_i^{1}\langle \ve{c}_{1j_1},\ve{x}_1\rangle+c_i^{2}\langle
\ve{c}_{2j_2},\ve{x}_2\rangle+\ldots+c_i^{m}\langle
\ve{c}_{mj_m},\ve{x}_m\rangle = 1
$$
and
$$
c_i^{1}\langle \ve{c}_{1j_1'},\ve{x}_1\rangle+c_i^{2}\langle
\ve{c}_{2j_2},\ve{x}_2\rangle+\ldots+c_i^{m}\langle
\ve{c}_{mj_m},\ve{x}_m\rangle = 1
$$
Subtracting we get $c_i^1t_1(j_1)=c_i^1\langle
\ve{c}_{1j_1},\ve{x}_1\rangle = c_i^1\langle
\ve{c}_{1j_1},\ve{x}_1\rangle = c_i^1t_1(j_1')$. Since $c_i^1\neq
0$ (at least for one $i$) we get $t_1(j_1)=t_1(j_1')$.

Thus far we can simply write $t_i$ instead of $t_i(j_i)$. Then
$\ve{t}=(t_1,\ldots,t_m)\in P$. As a consequence, $\langle
\ve{c}_{ij_i},\frac{\ve{x}_i}{t_i}\rangle=1$ for each $i$ and
$j_i$. Then
$\ve{x}=(t_1\frac{\ve{x}_1}{t_1},t_2\frac{\ve{x}_2}{t_2},\ldots,
t_m\frac{\ve{x}_m}{t_m})$ where $\ve{t}\in P$ and
$\frac{\ve{x}_i}{t_i}\in P_i$. This means $\ve{x}\in
P(P_1,\ldots,P_m)$ by definition.
\end{proof}

\ex\label{examplIterJsPoly} Let $P\subseteq \Rg^m$ be a natural
stochastic polytope (with $m$ facets) defined by relations
$\{\langle \ve{c}_i,\ve{x}\rangle=1\}$ and
$\Delta_{[l_i]}\subseteq \Ro^{l_i}$ a standard simplex given by
$\{x_1+\ldots+x_{l_i}=1\}$. The polytope
$P(l_1,\ldots,l_m)=P(\triangle_{[l_1]},\ldots,\triangle_{[l_m]})\subseteq
\Ro^{\sum l_i}$ is called the \textbf{iteration} of the polytope
$P$. It is given in $\Rg^{\sum l_i}$ by the system of affine
relations

\begin{equation}
c_i^1(x_{11}+\ldots+x_{1l_1})+c_i^2(x_{21}+\ldots+x_{2l_2})+\ldots+
c_i^m(x_{m1}+\ldots+x_{ml_m})=1.
\end{equation}

If $P$ is simple then $P(l_1,\ldots,l_m)$ is simple as well (see
section \ref{sectCompHomotopySpheres} or the work \cite{BBCGit}).
Such polytopes, their quasitoric manifolds and moment-angle
complexes were studied in \cite{BBCGit}. For the particular case
$P(l,\ldots,l)$, $l>0$ we use the notation $lP$.

\rem In section \ref{sectCompHomotopySpheres} we will show that
for natural stochastic polytopes the operation $P(P_1,\ldots,P_m)$
depends up to combinatorial equivalence only on the combinatorial
type of polytopes. Since each polytope has a natural stochastic
representation we can view the composition as the operation on
combinatorial polytopes.

\begin{prop}[Associativity law for the composition of
polytopes]\label{propAssocPolyIter} Let $P$ be a stochastic
polytope in $\Rg^m$, $P_1,\ldots,P_m$ be stochastic polytopes in
$\Rg^{l_1},\ldots,\Rg^{l_m}$ respectively and
$$P_{11},\ldots,P_{1l_1}, P_{21},\ldots,P_{2l_2}, \ldots,
P_{m1},\ldots,P_{ml_m}$$ --- stochastic polytopes as well. Then
\begin{multline}P(P_1(P_{11},\ldots,P_{1l_1}),\ldots,P_m(P_{m1},\ldots,P_{ml_m}))
=\\
P(P_1,\ldots,P_m)(P_{11},\ldots,P_{1l_1},\ldots,P_{m1},\ldots,P_{ml_m}).
\end{multline}
\end{prop}

The proof follows easily from the definition
\ref{definComposedPolytope}.

\rem It can be seen that $P(\pt,\ldots,\pt) = \pt(P) = P$, where
$\pt = \triangle_{[1]}$ is a point. Thus far the set of all
stochastic polytopes carries the structure of an operad, where the
polytope in $\Rg^m$ is viewed as $m$-adic operation and the
composition is given by the composition of polytopes described
above. Proposition \ref{propAssocPolyIter} expresses the
associativity condition for the operad and the polytope $\pt$ is
the ``identity'' element. Natural stochastic polytopes form a
suboperad by proposition \ref{propDescrCompPolyInRelations}.

\section{Composition of simplicial
complexes}\label{sectCompSimpComp}

Consider a simplicial complex $K$ on $m$ vertices and a set of
topological pairs $\{(X_i,A_i)\}_{i\in[m]}$, $A_i\subseteq X_i$.
For a simplex $I\in K$ let $V_I$ be the subset of
$X_1\times\ldots\times X_m$ given by $V_I = C_1\times\ldots\times
C_m$, where $C_i = X_i$ if $i\in I$ and $C_i = A_i$ if $i\notin
I$. The space
$$
\Z_K(\underline{(X_i,A_i)}) = \bigcup\limits_{I\in K}V_I\subseteq
\prod\limits_iX_i
$$
is called the \textbf{polyhedral product} of pairs $(X_i,A_i)$
defined by $K$.

\ex The motivating examples of polyhedral products are
moment-angle complexes $\mathcal{Z}_K(D^2,S^1)$, real moment-angle
complexes $\Z_K(D^1,S^0)$ and Davis--Januszkiewicz spaces $DJ(K) =
\Z_K(CP^\infty,\pt)$ (see \cite{BP2}). Another series of examples
is given by wedges $\bigvee_{\alpha} X_{\alpha}\cong
\Z_{\Delta_{[m]}^{(0)}}(\underline{(X_\alpha,\pt)})$, fat wedges
$\Z_{{\partial\Delta_{[m]}}}(\underline{(X_\alpha,\pt)})$ and
generalized fat wedges
$\Z_{\Delta_{[m]}^{(k)}}(\underline{(X_\alpha,\pt)})$. The spaces
of the form $\Z_K(\underline{(X_\alpha,\pt)})$ were studied in
\cite{An}. The most general situation
$\Z_K(\underline{(X_i,A_i)})$ was defined and studied by A. Bahri,
M. Bendersky, F. R. Cohen and S. Gitler in \cite{BBCG} from the
homotopy point of view.
\\
\\
The very natural thing is to substitute the topological product in
the definition of a polyhedral product by any other operation on
topological spaces. Thus far we can get \textbf{polyhedral smash
product} $\Z_K^{\wedge}(\underline{(X_i,A_i)})$ \cite{BBCG} and
\textbf{polyhedral join} $\Z_K^{\ast}(\underline{(X_i,A_i)})$ as
defined below.

\begin{defin}\label{definPolyhedralJoin}
Let $\{(X_i,A_i)\}_{i\in[m]}$ be topological pairs and $K$ a
simplicial complex on $[m]$. For each simplex $I\in K$ consider a
subset $V_I\subseteq X_1\ast\ldots\ast X_m$ of the form $V_I =
C_1\ast\ldots\ast C_m$, where $C_i = X_i$ if $i\in I$ and $C_i =
A_i$ if $i\notin I$. The space
$$
\Z_K^{\ast}(\underline{(X_i,A_i)}) = \bigcup\limits_{I\in
K}V_I\subseteq \divideontimes X_i
$$
is called the polyhedral join of pairs $(X_i,A_i)$.
\end{defin}

\begin{obser}
If $X_i$ is a simplicial complex and $A_i$ its simplicial
subcomplex, the space $\Z_K^{\ast}(\underline{(X_i,A_i)})$ has a
canonical simplicial structure. So far the polyhedral join is well
defined on the category of simplicial complexes as opposed to
polyhedral product.
\end{obser}

Let $K$ be a simplicial complex on the set $[m]$. It can be
considered as a subcomplex of $\Delta_{[m]}$ --- the simplex on
the set $[m]$, so far there is a pair $(\Delta_{[m]},K)$.

\begin{defin}\label{definComposedSimpComp}
Let $K$ be a simplicial complex on the set $[m]$ and $K_i$ a
simplicial complex on the set $[l_i]$ for each $i\in [m]$. The
simplicial complex $K(K_1,\ldots,K_m) =
\Z_K^{\ast}(\underline{(\Delta_{[l_i]},K_i)})$ will be called the
\textbf{composition} of $K$ with $K_i, i\in [m]$.
\end{defin}

Now we define the composition of simplicial complexes in purely
combinatorial terms. Let $K$ be a simplicial complex on $m$
vertices, which are possibly ghost. Let $K_1,\ldots,K_m$ be
simplicial complexes on the sets $[l_1],\ldots,[l_m]$ (ghost
vertices are allowed as well). Then $K(\underline{K_i})$ is a
simplicial complex on the set $[l_1]\sqcup\ldots\sqcup[l_m]$
defined by the following condition: the set
$I=I_1\sqcup\ldots\sqcup I_m, I_i\subseteq [l_i]$ is the simplex
of $K(K_1,\ldots,K_m)$ whenever $\{i\in[m]\mid I_i\notin K_i\}\in
K$.

\begin{figure}[h]
\begin{center}
\includegraphics[scale=0.3]{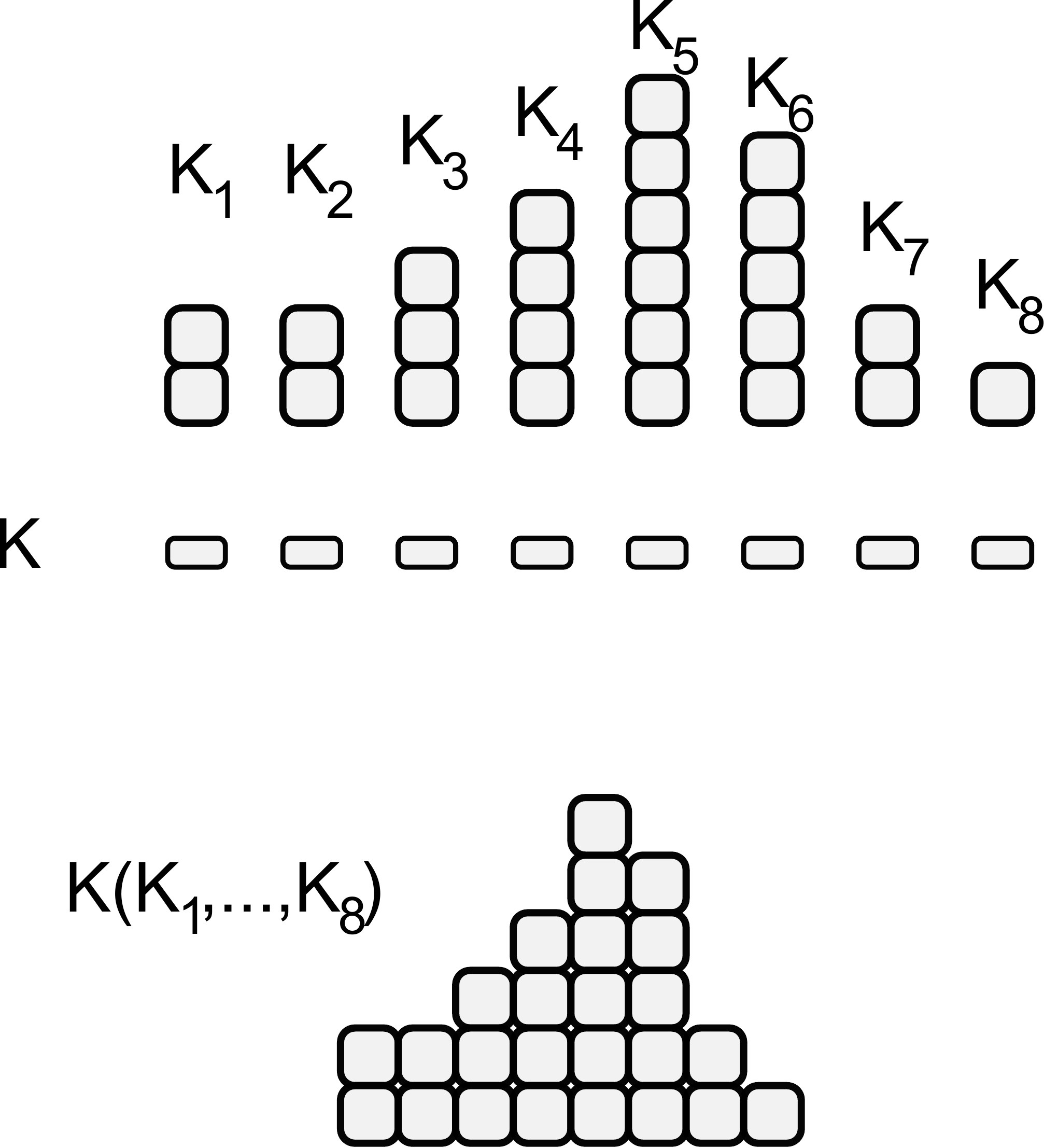}
\end{center}
\caption{Vertices of $K(K_1,\ldots,K_m)$}\label{FigComposedVert}
\end{figure}

The process of constructing the complex $K(\Ks)=K(K_1,\ldots,K_m)$
is depicted on figures \ref{FigComposedVert} and
\ref{FigComposedSimp}. The set of vertices of $K(\Ks)$ is the
union of vertices of $K_i$, which can be depicted by a simple
diagram (fig. \ref{FigComposedVert}). To construct the simplex of
$K(\Ks)$ we fix any simplex $J\in K$ and take full subcomplex
$\Delta_{[l_i]}$ (or any of its faces) for $i\in J$ and any
simplex $I_i\in K_i$ for each $i\notin J$. The union of these sets
gives a simplex of $K(\Ks)$ (fig. \ref{FigComposedSimp}). All
simplices $I\in K(\Ks)$ can be constructed by such procedure. This
approach to the construction of $K(\Ks)$ will be discussed in
section \ref{sectCompHomotopySpheres} in more detail.

\begin{figure}[h]
\begin{center}
\includegraphics[scale=0.3]{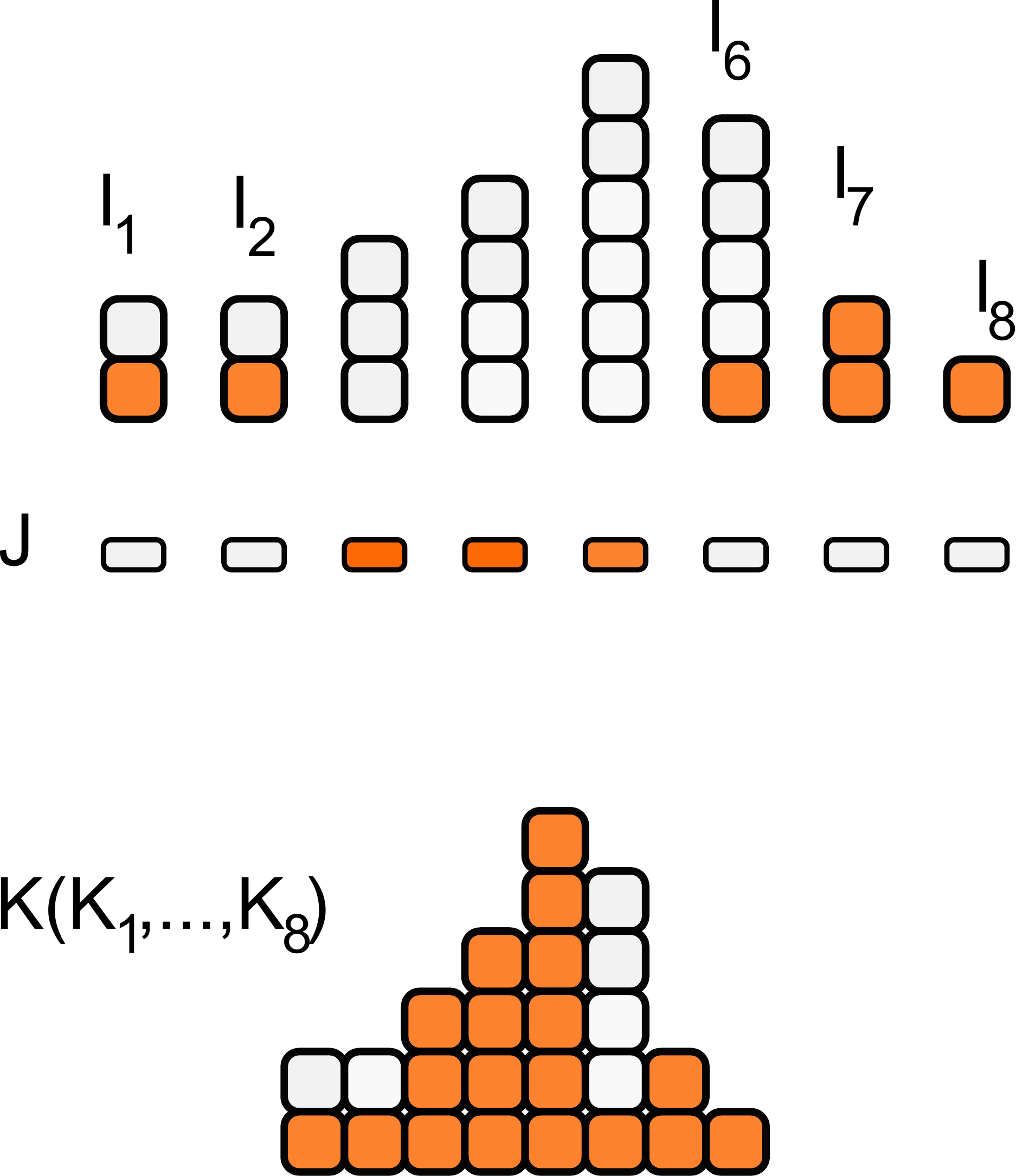}
\end{center}
\caption{Simplex of $K(K_1,\ldots,K_m)$}\label{FigComposedSimp}
\end{figure}

Let $o^{l}$ be the simplicial complex on $l>0$ vertices which has
no nonempty simplices. It means that all its vertices are ghost.
We formally set $\partial\Delta_{[1]}=o^1$.

By remark \ref{remarkNerveStochasticDescr} we can set $K_{\pt} =
o^1$ since the polytope $\pt=\triangle_{[1]}$ is defined by
$\Rg\cap\{x\in\Ro\mid x=1\}$ and does not intersect the hyperplane
$\{x=0\}$.

\ex\label{examplO1} We have by definition $K(o^1,\ldots,o^1) = K$
and $o^1(K) = K$.

\ex\label{examplOl} The complex $K(o^l,o^1,\ldots,o^1)$. Let $v_1$
be the first vertex of $K$. Then $K(o^l,o,\ldots,o)$ can be
described by the following procedure: the vertex $v_1$ is replaced
by a simplex $I_1=\{v_1^{1},\ldots,v_1^{l}\}$ and simplices $I\in
K$, containing $v_1$ are blown up to simplices
$(I\setminus\{v_1\})\sqcup I_1$. Therefore, $K(o^l,o,\ldots,o) =
K_{[m]\setminus v_1}\cup (\link_Kv_1)\ast I_1$.
\ex\label{examplOofKs} $o^m(K_1,\ldots,K_m) = K_1\ast\ldots\ast
K_m$.
\\
\\
Next statement provides a connection between the composition of
polytopes (in the natural stochastic case) and the composition of
simplicial complexes.

\begin{prop}\label{propFunctorOperad}
Let $P,P_1,\ldots, P_m$ be natural stochastic polytopes. Then
$$K_{P(P_1,\ldots,P_m)} = K_P(K_{P_1},\ldots,K_{P_m}).$$
\end{prop}

\begin{proof}
We need a technical lemma. Recall from section
\ref{sectCompPolytopes} that for $\ve{x}\in \Rg^m$,
$\ts(\ve{x})=\{i\in [m]\mid x_i=0\}$.

\begin{lemma}\label{lemmaSolutionForGivenY}
Let $Q$ be a polytope given by $\Rg^m\cap \{\sum c_i^jx_j=1,
i=1,\ldots,m-n\}$ with $c_i^j>0$. Fix $y\in\Ro$. If $\ve{x}\in
\Rg^m$ is the solution to the system of equations $\sum
c_i^jx_j=y$, then $y\geqslant 0$ and either $\zer(\ve{x})\in K_Q$
if $y>0$ or $\ve{x}=\ve{0}$ if $y=0$.
\end{lemma}

\begin{proof}
If $y=0$, the statement is evident since $c_i^j>0$ and $\ve{x}$
should be nonnegative. If $y>0$ consider the point $\ve{x}/y$. It
can be seen that $\ve{x}/y\in Q$ and
$\zer(\ve{x}/y)=\zer(\ve{x})$. Therefore $\zer(\ve{x})\in K_Q$.
\end{proof}

Let $\ve{x}=(x_1,\ldots,x_l)\in\Ro^l$. If $\ve{x}\in P$, then
$\zer(\ve{x})\in K_P$. Vice-versa, if $I\in K_P$ then there exists
$\ve{x}\in P$ such that $I\subseteq\zer(\ve{x})$ (remark
\ref{remarkNerveStochasticDescr}).

It can be seen that both complexes $K_{P(P_1,\ldots,P_m)}$ and
$K_P(K_{P_1},\ldots,K_{P_m})$ have the same set of vertices
$[l_1]\sqcup\ldots\sqcup[l_m]$. Denote $l_1+\ldots+l_m$ by
$\Sigma$. Let
\begin{equation*}
\ve{x}=(x_{11},\ldots,x_{1l_1},x_{21},\ldots,x_{2l_2},\ldots,x_{m1},\ldots,x_{ml_m})\in\Rg^{\Sigma}=
(\ve{x}_1,\ldots,\ve{x}_m)
\end{equation*}
be the point of $P(P_1,\ldots,P_m)$. Then for the point $\ve{x}$
of $P(P_1,\ldots,P_m)$ we have
$$\left\langle\ve{c}_i,(\langle\ve{c}_{1j_1}, \ve{x}_1\rangle,
\langle\ve{c}_{2j_2}, \ve{x}_2\rangle,\ldots,\langle\ve{c}_{mj_m},
\ve{x}_m\rangle)\right\rangle=1.$$ Denote $\langle\ve{c}_{sj_s},
\ve{x}_s \rangle$ by $t_s$ (it does not depend on $j_1,\ldots,
j_m$ --- see proof of proposition
\ref{propDescrCompPolyInRelations}) and set
$\ve{t}=(t_1,\ldots,t_m)$. By observation \ref{obserNonnegCoeffs}
we may assume $t_s\geqslant 0$. Therefore, $\zer(\ve{t})\in K_P$.
For all $s\in[m]$ we have an alternative:
\begin{itemize}
\item If $s\in \zer(\ve{t})$, then $t_s=0$ and
$\langle \ve{c}_{sj_s},\ve{x}_s\rangle = 0$. Then
$\ve{x}_s=\ve{0}$ by lemma \ref{lemmaSolutionForGivenY}.

\item If $s\notin\zer(\ve{t})$, then $t_s\neq 0$ and $\langle \ve{c}_{sj_s},\ve{x}_s\rangle =
t_s>0$. Then by lemma \ref{lemmaSolutionForGivenY}
$\zer(\ve{x}_s)\in K_{P_s}$.
\end{itemize}

Therefore, $\zer(\ve{x})\in K_P(K_{P_1},\ldots,K_{P_m})$.
Preceding arguments show that if $I\in K_{P(P_1,\ldots,P_m)}$,
then $I\in K_P(K_{P_1},\ldots,K_{P_m})$. Now let $J\in
K_P(K_{P_1},\ldots,K_{P_m})$, $J=A_1\sqcup\ldots\sqcup A_m$, where
$A_s\subseteq[l_s]$. We need to show that there exists a point
$\ve{x}\in P(P_1,\ldots,P_m)$ such that $J\in\zer(\ve{x})$.

By definition there exists a simplex $I\in K_P$ such that $s\notin
I$ implies $A_s\in K_{P_s}$. There exists a point
$\ve{t}=(t_1,\ldots,t_m)\in P$ such that $I\subseteq\zer(\ve{t})$.
Also for each $s$ there exist solutions to the system of equations
$\{\langle \ve{c}_{s,j_s},\ve{x}_s \rangle =
t_s\}_{j_s=1,\ldots,l_s}$ such that $A_s\subseteq\zer(\ve{x}_s)$
if $t_s\neq 0$ and $\ve{x}_s=\ve{0}$ if $t_s=0$ (equiv.
$\zer(\ve{x}_s)=[l_s]\supseteq A_s$). Then the nonnegative
solution to the system of equations

\begin{equation*}\left\langle\ve{c}_i,(\langle\ve{c}_{1j_1},
\ve{x}_1\rangle, \langle\ve{c}_{2j_2},
\ve{x}_2\rangle),\ldots,\langle\ve{c}_{mj_m},
\ve{x}_m\rangle)\right\rangle=1
\end{equation*}
is given by $\ve{x}=(\ve{x}_1,\ve{x}_2,\ldots,\ve{x}_m)$, where
$\zer(\ve{x}) =
\zer(\ve{x}_1)\sqcup\ldots\sqcup\zer(\ve{x}_m)\supseteq J$. This
concludes the proof.
\end{proof}

\begin{cor}
If $P, P_1,\ldots,P_m$ are combinatorially equivalent to
$Q,Q_1,\ldots,Q_m$ respectively and all the polytopes are natural
stochastic, then $P(P_1,\ldots,P_m)$ is combinatorially equivalent
to $Q(Q_1,\ldots,Q_m)$. Therefore, $P(P_1,\ldots,P_m)$ can be
viewed as a well-defined operation on combinatorial polytopes.
\end{cor}

\ex\label{examplSimplicialWedge} A nontrivial example of the
composition is the \textbf{iterated simplicial wedge} construction
as defined in \cite{BBCGit}. Let $K$ be a simplicial complex on
$m$ vertices and $(l_1,\ldots,l_m)$
--- an array of natural numbers. Consider the simplicial complex
$K(l_1,\ldots,l_m) =
K(\partial\Delta_{[l_1]},\ldots,\partial\Delta_{[l_m]})$.
\\
\\
If $P$ is a polytope, then $K_P(l_1,\ldots,l_m) =
K_P(\partial\Delta_{[l_1]},\ldots,\partial\Delta_{[l_m]}) =
K_P(K_{\triangle_{[l_1]}},\ldots,K_{\triangle_{[l_m]}}) =
K_{P(\triangle_{[l_1]},\ldots, \triangle_{[l_m]})} =
K_{P(l_1,\ldots,l_m)}$ by proposition \ref{propFunctorOperad}. In
section \ref{sectCompHomotopySpheres} we will show that for every
$m$-tuple $(l_1,\ldots,l_m)$ simplicial complex $K$ is a
combinatorial sphere whenever $K(l_1,\ldots,l_m)$ is a
combinatorial sphere. Then $P$ is simple whenever
$P(l_1,\ldots,l_m)$ is simple (see example
\ref{examplPsimpleKPsphere}).

\begin{prop}[Associativity law for the composition of simplicial
complexes]\label{propAssocSimpIter} Let $K$ be a simplicial
complex on $m$ vertices, $K_1,\ldots,K_m$ be simplicial complexes
on $l_1,\ldots,l_m$ vertices respectively and
$K_{11},\ldots,K_{1l_1},K_{21},\ldots,K_{2l_2},\ldots,
K_{m1},\ldots,K_{ml_m}$ --- simplicial complexes on sets
$[r_{sj_s}]$ of vertices. Then
\begin{multline}K(K_1(K_{11},\ldots,K_{1l_1}),\ldots,K_m(K_{m1},\ldots,K_{ml_m}))
=\\
K(K_1,\ldots,K_m)(K_{11},\ldots,K_{1l_1},\ldots,K_{m1},\ldots,K_{ml_m})
\end{multline}
as the complexes on the set $\bigsqcup_{s,j_s}[r_{sj_s}]$.
\end{prop}

\begin{proof}
Both complexes have the same set of vertices
$V=([r_{11}]\sqcup\ldots\sqcup[r_{1l_1}])\sqcup\ldots\sqcup([r_{m1}]\sqcup\ldots\sqcup[r_{ml_m}])$
Let $A$ be the subset of $V$ so $A = (A_{11}\sqcup\ldots\sqcup
A_{1l_1})\sqcup\ldots\sqcup(A_{m1}\sqcup\ldots\sqcup A_{ml_m})$,
where $A_{sj_s}\subseteq [r_{sj_s}]$. The chain of equivalent
conditions is written below.
\begin{multline}
A\in K(K_1(K_{11},\ldots,K_{1l_1}),\ldots,K_m(K_{m1},\ldots,K_{ml_m})) \Leftrightarrow\\
\exists I\in K \forall s\notin I \colon (A_{s1}\sqcup\ldots\sqcup
A_{sl_s})\in
K_s(K_{s1},\ldots,K_{sl_s})\Leftrightarrow\\
\exists I\in K \forall s\notin I \exists I_s\in K_s \forall
i_s\notin I_s \colon A_{si_s}\in K_{si_s}
\Leftrightarrow\\
\exists J\in K(K_1,\ldots,K_m) \forall s \forall i_s\in
[l_s]\setminus J \colon A_{si_s}\in K_{si_s} \Leftrightarrow\\
A\in
K(K_1,\ldots,K_m)(K_{11},\ldots,K_{1l_1},\ldots,K_{m1},\ldots,K_{ml_m}).
\end{multline}

This finishes the proof.
\end{proof}

\rem As in the case of polytopes simplicial complexes form an
operad. The simplicial complex $K$ on $m$ vertices can be viewed
as an $m$-adic operation. The ``identity operation'' is given by
the complex $o^1$ (see example \ref{examplOl}) since
$K(o^1,\ldots,o^1) = o^1(K) = K$.

\begin{cor}
The composition can be constructed by steps. More precisely, let
$K_i$ be the complex on $[l_i]$, then $$K(K_1,\ldots,K_m) =
K(o^1,\ldots, K_i,\ldots,o^1)(K_1,\ldots, K_{i-1},
\underbrace{o^1,\ldots,o^1}_{l_i},K_{i+1},\ldots,K_m).$$
\end{cor}

\begin{cor}[{\cite[sect.2]{BBCGit}}]\label{corIterSphereIsSphere}
Let $l_i$ be natural numbers. Then $$K(l_1,\ldots,l_m) =
K(1,\ldots, l_i,\ldots,1)(l_1,\ldots, l_{i-1},
\underbrace{1,\ldots,1}_{l_i},l_{i+1},\ldots,l_m).$$
\end{cor}

One can ``blow up'' vertices step by step. The operation
$K(l_1,1,1,\ldots,1)$ can be described geometrically
\cite{BBCGit}, \cite{PB}:
$$
K(l,1,1,\ldots,1) = K_{[m]\setminus \{1\}}\ast \partial
\Delta_{[l]}\cup (\link_K\{1\})\ast\Delta_{[l]}.
$$
The figure \ref{FigPentagon} illustrates the situation when $K$ is
the boundary of 5-gon and $l=2$.

\begin{figure}[h]
\begin{center}
\includegraphics[scale=0.25]{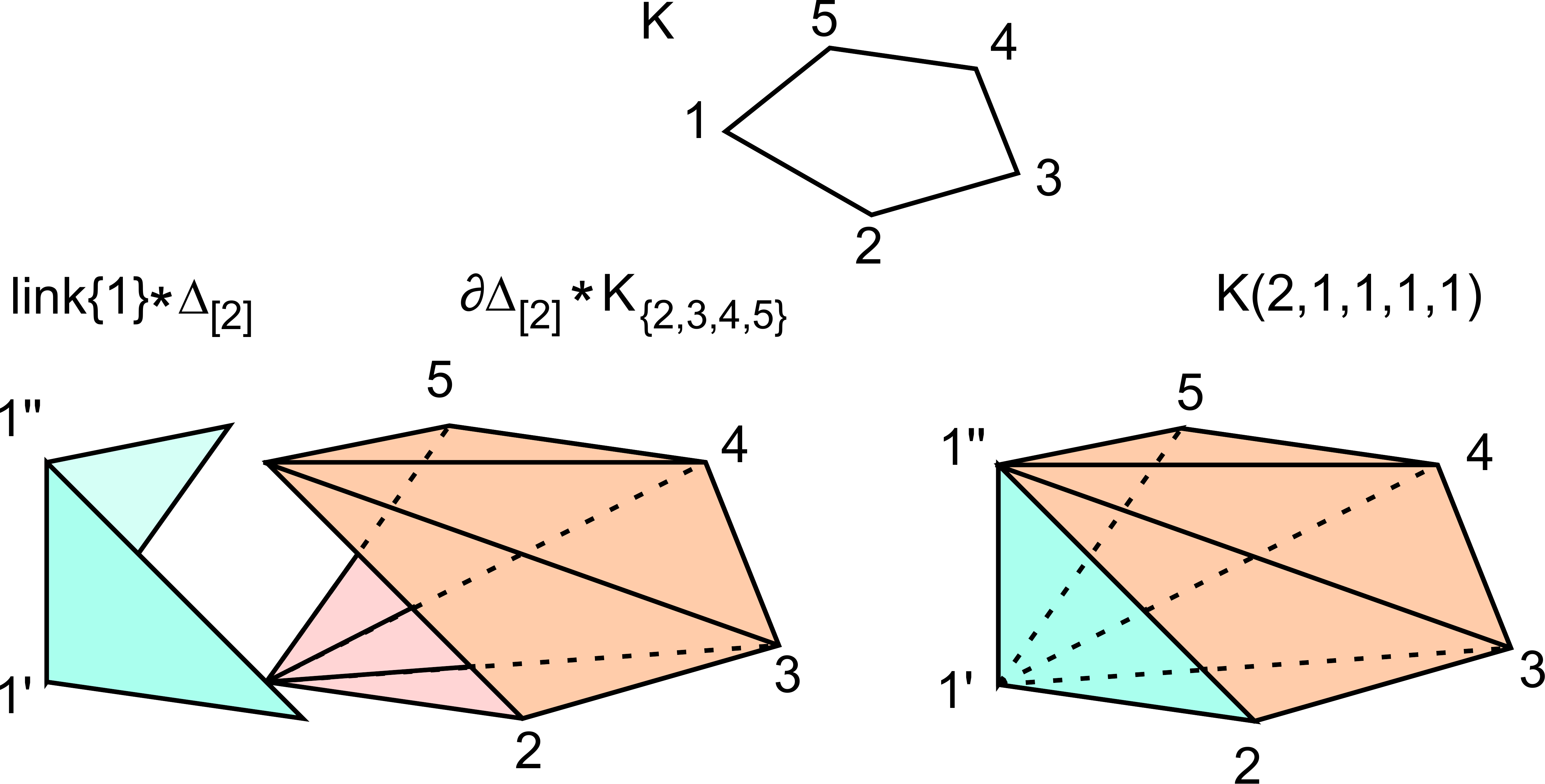}
\end{center}
\caption{Doubling a vertex in a boundary of a
pentagon}\label{FigPentagon}
\end{figure}

It can be directly checked that
$K(l,1,1,\ldots,1)\cong_{PL}\Sigma^{l-1}K\cong_{PL} K\ast
\partial\Delta_{[l]}$. In the case when $K$ is the boundary of
simplicial polytope the complex $K(l,1,1,\ldots,1)$ is also the
boundary of a polytope \cite[Th.2.3]{BBCGit}. Indeed, if
$K=\partial Q$, then $K=K_{Q^*}$ for dual simple polytope $Q^*$,
then $K(l_1,\ldots,l_m)=K_{Q^*(l_1,\ldots,l_m)} = \partial
(Q^*(l_1,\ldots,l_m))^*$. Then, using corollary
\ref{corIterSphereIsSphere} inductively, we get the following.

\begin{cor}
If $K$ is a simplicial sphere (boundary of a simplicial polytope,
triangulated topological sphere, homological sphere) then so is
$K(l_1,\ldots,l_m)$ for each array $l_1,\ldots,l_m$ of natural
numbers.
\end{cor}

There is the converse question: for which $K,K_1,\ldots,K_m$ the
composed complex $K(K_1, \ldots, K_m)$ is a sphere (in any sense)?
We postpone this question till section
\ref{sectCompHomotopySpheres}.

\section{Polyhedral products given by composed
complexes}\label{sectPolyProducts}


In this section we describe the polyhedral products given by the
composed simplicial complexes.

\begin{prop}\label{propPolyProdOfComposed}
Let $K$ be a complex with $m$ vertices and let $\{K_i\}_{i\in[m]}$
be simplicial complexes with $l_i$ vertices. Consider topological
pairs, indexed by the elements of the set $\bigsqcup_{i}[l_i]$:
$$(X_{11},A_{11}), \ldots,(X_{1l_1},A_{1l_1}), \ldots,
(X_{m1},A_{m1}), \ldots,(X_{ml_m},A_{ml_m}).$$ For each $i\in[m]$
let $Y_i=X_{i1}\times X_{i2}\times\ldots\times X_{il_i}$ and $Z_i
= \Z_{K_i}(\underline{(X_{ij},A_{ij})}_{j\in[l_i]})\subseteq Y_i$.
Then
$$
\Z_{{K(\Ks)}}(\underline{(X_{ij},A_{ij})}_{\substack{i\in[m]\\j\in[l_i]}})
= \Z_K(\underline{(Y_i,Z_i)}_{i\in[m]})
$$
as subsets of $\prod_{i,j}X_{ij} = \prod_iY_i$.
\end{prop}

The proof is similar to the proposition \ref{propAssocSimpIter}
and rather straightforward.

\ex In the case $K_i=\partial\Delta_{[l_i]}$ the proposition
\ref{propPolyProdOfComposed} coincides with \cite[Theorem
7.2]{BBCGit}.

\ex Let $K, K_1,\ldots,K_m$ be as before. Consider the set of
pairs $\{(X_{ij},A_{ij})\}_{\substack{i\in[m]\\j\in[l_i]}}$, where
$(X_{ij},A_{ij}) = \{(D^2,S^1)\}$ for each $i\in[m]$ and $j\in
[l_i]$. Then
$$\Z_{K(\Ks)}(D^2,S^1)=\Z_K(\underline{(D^{2l_i},
\Z_{K_i})}_{i\in[m]}).$$

\ex As the particular case of the previous example consider $K_i =
\partial \Delta_{[l_i]}$. Then
$$
\Z_{K(l_1,\ldots,l_m)}(D^2,S^1)=\Z_K(\underline{(D^{2l_i},S^{2l_i-1})}_{i\in
[m]})
$$
since $\Z_{\partial\Delta_{[l_i]}}(D^2,S^1) = S^{2l_i-1}$ which
coincides with \cite[Corollary 7.3]{BBCGit}.

The similar statement holds for real moment-angle complexes
$$
\Z_{K(l_1,\ldots,l_m)}(D^1,S^0)=\Z_K(\underline{(D^{l_i},S^{l_i-1})}_{i\in
[m]})
$$

In particular
$$
\Z_{2K}(D^1,S^0) = \Z_{K(2,\ldots,2)}(D^1,S^0)=\Z_K(D^2,S^1).
$$

This idea was used by Yu.Ustinovsky in his work \cite{UstTR} to
prove the toral rank conjecture for moment-angle spaces.

\section{Combinatorial and topological properties of composed complexes}\label{sectCompHomotopySpheres}

In \cite{BBCG} was proved that
$\Z^{\wedge}_K(\underline{(X_i,A_i)})\simeq \Sigma(K\wedge
A_1\wedge\ldots\wedge A_m)$ if $X_i\simeq\pt$ for each $i\in[m]$.
The following statement can be proved by the similar argument.

\begin{prop}
Let $(X_i,A_i)$ be topological pairs for $i\in[m]$ with $X_i$
contractible. Then for every $K$ on $[m]$ the space
$\Z^{\ast}_K(\underline{(X_i,A_i)})$ is homotopy equivalent to
$A_1\ast\ldots\ast A_m\ast K$.
\end{prop}

\begin{proof}
Let $\cat(K)$ be a small category associated to simplicial complex
$K$. The objects of $\cat(K)$ are the simplices of $K$ and
morphisms are inclusions.

Define a functor $\Psi\colon\cat(K)\to\Top$. For each $I\in
\Ob\cat(K)$ let $\Psi(I)$ be the space $U_I = B_1\ast
B_2\ast\ldots\ast B_m$, where $B_i = X_i$ if $i\in I$ and
$B_i=A_i$ otherwise. The morphism $\Psi(I\hookrightarrow I')$ is
given by the natural inclusion $U_I\hookrightarrow U_{I'}$.

Then $\colim \Psi\cong \Z^{\ast}_K(\underline{(X_i,A_i)})$. All
the maps in the diagram $\Psi$ are closed cofibrations. Therefore,
the projection lemma (see, e.g. \cite[proposition 3.1]{WZZ})
implies $\colim\Psi \simeq \hocolim\Psi$.

Consider the diagram $\Phi\colon\cat(K)\to \Top$ given by
$\Phi(\varnothing) = A_1\ast\ldots\ast A_m$, $\Phi(I)=\pt$ if
$I\neq\varnothing$. The values of $\Phi$ on the morphisms
$I\subseteq I'$ are defined in the unique way.

1) For each $I\in\Ob\cat(K)$ there is a homotopy equivalence
$h_I\colon\Psi(I)\to\Phi(I)$. Indeed, for $I=\varnothing$ we have
$\Psi(I) = A_1\ast\ldots\ast A_m = \Phi(I)$ so $h_{\varnothing}$
can be chosen to be an identity map. For $I\neq\varnothing$ we
have $\Psi(I) = B_1\ast\ldots\ast B_m$ where at least one set
$B_i$ is equal to $X_i$ thus contractible. Therefore the whole
join $B_1\ast\ldots\ast B_m$ is contractible. Then the unique map
$h_{I}\colon \Psi(I)\to\Phi(I) = \pt$ is a homotopy equivalence.

2) The maps $h_I\colon \Psi(I)\to\Phi(I)$ are coherent, therefore,
$\hocolim \Psi\simeq\hocolim \Phi$.

3) $\hocolim\Phi \simeq \Phi(\varnothing)\ast K \cong K\ast
A_1\ast\ldots\ast A_m$ (see \cite[lemma 3.4]{WZZ}). This fact can
be deduced from the constructive definition of a homotopy colimit.

4) The sequence of equivalences
$$\Z_K^{\ast}(\underline{(X_i,A_i)})\cong \colim\Psi\simeq\hocolim\Psi\simeq
\hocolim\Phi=K\ast A_1\ast\ldots\ast A_m$$
completes the proof.
\end{proof}

\begin{cor}\label{corComposIsJoin} For any simplicial complexes
$K,K_1,\ldots, K_m$ with nonempty sets of vertices we have a
homotopy equivalence $$K(\Ks)\simeq K\ast K_1\ast
K_2\ast\ldots\ast K_m.$$
\end{cor}

\begin{cor}\label{corComposSpheresIsSphere}
If $K\simeq S^{n-1}$, $K_i\simeq S^{n_i-1}$ then $K(\Ks)\simeq
S^{n+n_1+\ldots+n_m-1}$.
\end{cor}

\begin{cor}\label{corDimOfPcomposedHomol}
Let $P,P_1,\ldots,P_m$ be the polytopes of dimensions
$n,n_1,\ldots,n_m$. Then $\dim P(\Ps) = n+n_1+\ldots+n_m$.
\end{cor}

\begin{proof}
$K_Q\simeq S^{n-1}$ if $\dim Q = n$ for any convex polytope $Q$
(see section \ref{sectPrelim}). Therefore, $S^{\dim P(\Ps)-1}
\simeq K_{P(\Ps)} = K_P(K_{P_1},\ldots,K_{P_m})\simeq K_P\ast
K_{P_1}\ast\ldots\ast K_{P_m}\simeq S^{n+n_1+\ldots+n_m-1}$. Then
$\dim P(\Ps) = n+n_1+\ldots+n_m$.
\end{proof}

\ex Let $K(o^{l_1},\ldots,o^{l_m})$ be the composition of $K$ with
``ghost complexes'' $o^{l}$. Then $K(o^{l_1},\ldots,o^{l_m})\simeq
K$ by corollary \ref{corComposIsJoin}. It can be seen from example
\ref{examplOl} as well.

\begin{thm}\label{thmSphereReduct}
Let $K(\Ks)$ be a simplicial sphere (resp. homological sphere) and
suppose $K$ does not have ghost vertices. Then $K$ is a simplicial
sphere (resp. homological sphere) and $K_i =
\partial\Delta_{[l_i]}$ for some $l_i>0$. If $K(\Ks)$ is the
boundary of a simplicial polytope (up to ghost vertices), then so
is $K$.
\end{thm}

\begin{proof}
We need a technical lemma which describes the links of simplices
in the composed complex $K(\Ks)$.

\begin{lemma}\label{lemmaLinkStructMain}
Suppose $K,K_1,\ldots,K_m$ are simplicial complexes on sets
$[m],[l_1],\ldots,[l_m]$. Let $A\in K(\Ks)$,
$A=A_1\sqcup\ldots\sqcup A_m$, $A_i\subseteq[l_i]$ and $J = \{i\in
[m]\mid A_i\notin K_i\}\in K$. Also let $\{i_1,\ldots,i_k\} =
[m]\setminus J$. For each $i\in J$ consider a set
$M_i=[l_i]\setminus A_i$ and a simplex $\Delta_{M_i}$ spanned by
this set. Then
$$
\link_{K(\Ks)}A = \link_KJ(\link_{K_{i_1}}A_{i_1},
\ldots,\link_{K_{i_k}}A_{i_k})\ast\left(\bigast_{i\in
J}\Delta_{M_i}\right).
$$
\end{lemma}

\begin{proof}
Both simplicial complexes have the same set of vertices
$\bigsqcup\limits_{i=1}^m([l_i]\setminus A_i)$. Let
$I=I_1\sqcup\ldots\sqcup I_m\in \link_{K(\Ks)}A$. By definition
this means $I\sqcup A\in K(\Ks)$. Equivalently,
$$B' = \{i\in[m]\mid I_i\sqcup A_i\notin K_i\}\in K.$$
Equivalently,
$$
B=\{i\in[m]\setminus J\mid I_i\sqcup A_i\notin K_i\}\sqcup J\in K,
$$
because $A_i\notin K_i$ yields $A_i\sqcup I_i\notin K_i$ and,
therefore, $J\subseteq B'$. Equivalently,
$$
B=\{i\in[m]\setminus J\mid I_i\sqcup A_i\notin K_i\}\in \link_KJ.
$$
Equivalently,
\begin{equation}\label{equatLinkFinCond}
B=\{i\in[m]\setminus J\mid I_i\notin \link_{K_i}A_i\}\in \link_KJ.
\end{equation}
So far
$$
I=\bigsqcup\limits_{i\notin J}I_i\sqcup\bigsqcup\limits_{i\in
J}I_i,
$$
where $\bigsqcup\limits_{i\notin J}I_i$ satisfies
\eqref{equatLinkFinCond}, therefore
$$
\bigsqcup\limits_{i\notin J}I_i\in
\link_{K}J(\link_{K_{i_1}}A_{i_1},\ldots,\link_{K_{i_k}}A_{i_k}).
$$
Since no conditions on $\bigsqcup\limits_{i\in J}I_i$ are imposed,
we get the required formula.
\end{proof}

Now let $K(\Ks)$ be a simplicial (resp. homological) sphere. Then
for any $A\in K(\Ks)$ the complex $\link_{K(\Ks)}A$ is a
simplicial (resp. homological) sphere as well. First of all, note
that $K_i\neq\Delta_{[l_i]}$. Indeed, otherwise $K(\Ks)\simeq
K\ast K_1\ast\ldots\ast K_m$ is contractible which contradicts the
assumption.

In what follows we use the notation of lemma
\ref{lemmaLinkStructMain}. Suppose there exists a number $j\in[m]$
for which $K_j$ has a nonsimplex $A_j\notin K_j$ such that
$A_j\neq[l_j]$. Consider the subset
$$A=\varnothing\sqcup\ldots\sqcup A_j\sqcup\ldots\sqcup\varnothing
\subseteq [l_1]\sqcup\ldots\sqcup[l_m].$$ Since $J=\{i\in[m]\mid
A_i\notin K_i\}=\{j\}\in K$ by the assumption, we have $A\in
K(\Ks)$. By lemma \ref{lemmaLinkStructMain} $\link_{K(\Ks)}A =
X\ast \Delta_{M_j}$, where $X$ is some complex and $\Delta_{M_j}$
is a simplex spanned by $[l_j]\setminus A_j\neq\varnothing$.
Therefore $\link_{K(\Ks)}A$ is contractible which contradicts the
assumption that it is a sphere.

Thus for each $i$ the only nonsimplices of $K_i$ are $[l_i]$. This
argument shows that $K_i = \partial\Delta_{[l_i]}$.

Let $I_i^{max}\in K_i=\partial\Delta_{[l_i]}$ be any maximal
simplex (facet) for each $i\in[m]$. Then $l_i-|I_i^{max}|=1$ and
$\link_{K_i}I_i^{max}=o^1$, the complex on one ghost vertex.

Consider the simplex $A = I_1^{max}\sqcup\ldots\sqcup I_m^{max}\in
K(\Ks)$. Since $J=\{i\in[m]\mid I_i\notin K_i\}=\varnothing$ we
have $\link_KJ = \link_K\varnothing = K$. Applying lemma
\ref{lemmaLinkStructMain} to the simplex $A$ we get
$$
\link_{K(\Ks)}A = \link_KJ(\link_{K_1}I_1^{max},\ldots,
\link_{K_m}I_m^{max}) = K(o,\ldots,o)=K,
$$
Since $\link_{K(\Ks)}A$ is a combinatorial (resp. homological)
sphere, so is $K$.

If $K(\Ks)$ is the boundary of a simplicial polytope then all its
links are the boundaries of simplicial polytopes. This gives the
last part of the proposition.
\end{proof}

\rem More general result can be obtained by the same arguments.
Let $K(\Ks)$ be a combinatorial (homological) sphere. Then

1) $K$ is a combinatorial sphere;

2) If $i$ is not a ghost vertex of $K$, then $K_i =
\partial\Delta_{[l_i]}$ for some $l_i>0$.

3) If $i$ is a ghost vertex of $K$, then $K_i$ is a combinatorial
(homological) sphere.

\begin{prop}\label{propCompOfPisSimpleThenPisSimple}
Let $K(\Ks)=K_Q$ for some simple polytope $Q$. Then there exists a
simple polytope $P$ and numbers $l_i>0$ such that $K=K_P$ and
$K_i=\partial\Delta_{[l_i]}$. The polytopes $Q$ and
$P(l_1,\ldots,l_m)$ are combinatorially equivalent.
\end{prop}

\begin{proof}
If $Q$ is simple $K_Q=\partial Q^*$. Therefore, if $K(\Ks)=K_Q$,
then $K(\Ks)$ is the boundary of a simplicial polytope and by
theorem \ref{thmSphereReduct} $K_i=\partial\Delta_{[l_1]}$ and $K$
is the boundary of a simplicial polytope. Then $K = K_P$ and $K_Q
= K(\Ks)=K_P(\partial\Delta_{[l_1]},\ldots,\partial\Delta_{[l_m]})
= K_{P(l_1,\ldots,l_m)}$ by proposition \ref{propFunctorOperad}.
This means that $Q$ and $P(l_1,\ldots,l_m)$ are combinatorially
equivalent (see section \ref{sectPrelim}).
\end{proof}

Proposition \ref{propFunctorOperad} motivated the assumption, that
the class of spherical nerve-complexes is closed under composition
unlike the class of simplicial spheres.

\begin{thm}\label{thmNerveCmpxCompos}
Let $K$ be the spherical nerve-complex of rank $n$ with $m$
vertices and $K_1,\ldots,K_m$ be spherical nerve-complexes on
$[l_1],\ldots,[l_m]$ of ranks $n_1,\ldots,n_m$ respectively. Then
$K(K_1,\ldots,K_m)$ is a spherical nerve-complex of rank
$n+n_1+\ldots+n_m$.
\end{thm}

\begin{proof}
We use the notation of section \ref{sectPrelim}. Let us describe
the set of maximal simplices $M(K(\Ks))$ and the set of their
intersections $F(K(\Ks))$. We have $I_1\sqcup\ldots\sqcup I_m \in
M(K(\Ks))$ iff there exist a simplex $I\in M(K)$, such that $I_j =
[l_j]$ for $j\in I$ and $I_j\in M(K_j)$ for $j\notin I$. Then
$I_1\sqcup\ldots\sqcup I_m \in F(K(\Ks))$ iff there exists $I\in
F(K)$, such that $I_j = [l_j]$ for $j\in I$ and $I_j\in F(K_j)$
for $j\notin I$. In this case we will say that $I$ is the support
of $I_1\sqcup\ldots\sqcup I_m$. Obviously $\varnothing\in
F(K(\Ks))$.

The poset $F(K(\Ks))$ is graded by the rank function
$$\rank_{F(K(\Ks))} (I_1\sqcup\ldots\sqcup I_m) =
\rank_{F(K_1)}'I_1+\ldots+\rank_{F(K_m)}'I_m+\rank_{F(K)}I,$$
where $\rank_{F(K_j)}'I_j = \rank_{F(K_j)}I_j$ if $I_j\in F(K_j)$
(that is $j\notin I$) and $\rank_{F(K_j)}'I_j = n_j$, the rank of
the nerve-complex $K_j$, if $I_j = V_j$ (in the case $j\in I$).

For a link of $I_1\sqcup\ldots\sqcup I_m$ with the support $I$ we
have
$$\link_{K(\Ks)}(I_1\sqcup\ldots\sqcup I_m) =
\link_KI\left(\underline{\left\{\link_{K_j}I_j\right\}}_{j\notin
I}\right).$$

By corollary \ref{corComposIsJoin}
\begin{multline}
\link_KI\left(\underline{\left\{\link_{K_j}I_j\right\}}_{j\notin
I}\right)\simeq \link_KI\ast \left(\bigast_{j\notin
I}\link_{K_j}I_j\right) \simeq\\ \simeq
S^{n-\rank_{F(K)}I-1}\ast\left(\bigast_{j\notin I}
S^{n_j-\rank_{F(K_j)}I_j-1}\right) \cong S^{n+\sum\limits_{j\notin
I}n_j - \rank_{F(K)}I - \sum\limits_{j\notin
I}\rank_{F(K_j)}I_j-1}.
\end{multline}
Since $\rank_{F(K_j)}'I_j = n_j$ in the case $j\in I$, by adding
and subtracting \linebreak$\sum_{j\in I}\rank_{F(K_j)}'I_j$ to the
dimension of a sphere in the last expression we get
\begin{multline}
n+\sum\limits_{j\notin I}n_j - \rank_{F(K)}I -
\sum\limits_{j\notin I}\rank_{F(K_j)}I_j-1 =\\=
n+\sum\limits_{j\in[m]}n_j-\sum\limits_{j\in[m]}\rank_{F(K_j)}'I_j-1
= n+\sum\limits_{j\in[m]}n_j -
\rank_{F(K(\Ks))}(I_1\sqcup\ldots\sqcup I_m)-1,
\end{multline}
and the statement follows.
\end{proof}

\section{Multigraded Betti numbers of the
compositions}\label{sectMultBettiNums}

In this section we review the definition of multigraded Betti
numbers of a simplicial complex $K$, and the Hochster formula
expressing multigraded Betti numbers as the ranks of cohomology
groups of full subcomplexes in $K$. Together with corollary
\ref{corComposIsJoin} Hochster formula will give an explicit
formula expressing multigraded Betti numbers of $K(\Ks)$ in terms
of multigraded Betti numbers of $K, K_1,\ldots,K_m$. In particular
cases this formula has very simple form and allows to find the
$h$-vectors of composed complexes.

Let $\ko$ be the ground field and $\ko[m] = \ko[v_1,\ldots,v_m]$
--- the ring of polynomials in $m$ indeterminates. The ring $k[m]$
has a $\Zo^m$-grading defined by $\deg v_i =
(0,\ldots,2,\ldots,0)$ with $2$ on the $i$-th place. The field
$\ko$ is given the $\ko[m]$-module structure via the epimorphism
$\ko[m]\to \ko$, $v_i\mapsto 0$.

Let $K$ be a simplicial complex on $m$ vertices. The
Stanley--Reisner algebra $\ko[K]$ is defined as a quotient algebra
$\ko[m]/I_{SR}$, where the Stanley--Reisner ideal $I_{SR}$ is
generated by square-free monomials $v_{\alpha_1}\ldots
v_{\alpha_k}$ corresponding to nonsimplices
$\{\alpha_1,\ldots,\alpha_m\}\notin K$. The algebra $\ko[K]$ has a
natural $\ko[m]$-module structure, given by quotient epimorphism
$\ko[m]\to \ko[m]/I_{SR}$. Since $I_{SR}$ is homogeneous ideal,
the module $\ko[m]$ is $\Zo^m$-graded.

Let $\ldots \to R^{-i}\to R^{-i+1}\to\ldots\to R^{-1}\to R^{0}\to
\ko[K]$ be a free resolution of the module $\ko[K]$ by
$\Zo^m$-graded $\ko[m]$-modules $R^{-i}$. We have $R^{-i} =
\bigoplus_{\bar{j}\in \Zo^m} R^{-i,\bar{j}}$. The Tor-module of a
complex $K$ therefore has a natural $\Zo^{m+1}$-grading:
$$
\Tor_{\ko[m]}(\ko[K],\ko) = \bigoplus\limits_{i\in
\Zo_{\geqslant}; \bar{j}\in\Zo^m}
\Tor^{-i,2\bar{j}}_{\ko[m]}(\ko[K],\ko).
$$
The multigraded Betti numbers of a complex $K$ are defined as the
dimensions of the graded components of the Tor-module:
$$
\beta_{\ko}^{-i,2\bar{j}}(K) = \dim_{\ko}
\Tor^{-i,2\bar{j}}_{\ko[m]}(\ko[K],\ko).
$$
These numbers depend on the ground field but we will omit $\ko$ to
avoid cumbersome notation. A combinatorial description of
multigraded Betti numbers is given by Hochster formula
\cite{Hoch,BP}.

\begin{thm}[{Hochster, \cite[Th.3.2.8]{BPnew}}]
For a simplicial complex $K$ on $m$ vertices and $\bar{j} =
(j_1,\ldots,j_m)\in \Zo^m$ there holds $\beta^{-i,2\bar{j}} = 0$
if $\bar{j}\notin \{0,1\}^m$. If $\bar{j}\in \{0,1\}^m$ and
$A=\{i\in[m]\mid j_i=1\}$, then
\begin{equation}\label{equatHochster}
\beta^{-i,2\bar{j}} = \Hr^{|A|-i-1}(K_{A};\ko),
\end{equation}
where $K_{A}$ is a full subcomplex of $K$ on the set of vertices
$A$. In this formula it is assumed that $\Hr^{-1}(\varnothing;\ko)
= \ko$.
\end{thm}

By this result the set of all multigraded Betti numbers is a
complete combinatorial invariant of a simplicial complex.

If $A\subseteq[m]$ we use the notation $\beta^{-i,2A}$ for the
number $\beta^{-i,2\bar{j}}$, where $\bar{j} = (j_1,\ldots,j_m)$,
$j_i=1$ if $i\in A$ and $j_i=0$ otherwise.

\rem The full subcomplex will be sometimes denoted by $K|_{A}$
instead of $K_{A}$, especially in the case when for $K$ stands a
complex with its own lower index.
\\
\\
Bigraded Betti numbers are defined by the formula
$\beta^{-i,2j}(K) = \sum_{|A|=j}\beta^{-i,2A}$. These numbers are
the dimensions of graded components of the Tor-algebra
$\Tor^{\ast,\ast}_{\ko[m]}(\ko[K],\ko)$ if we specialize the
$\Zo^m$-grading $(j_1,\ldots,j_m)$ to the $\Zo$-grading $\sum
j_i$.

To work with multigraded Betti numbers we construct their
generating functions, called beta-polynomials of $K$. Let
$$
\beta_K(s,\bar{t}) =
\beta_K(s,t_1,t_2,\ldots,t_m)=\sum\limits_{i\in \Zo,\bar{j}\in
\Zo^m} \beta^{-i,2\bar{j}}(K)s^{i}\bar{t}^{\bar{j}},
$$
where $\bar{t}^{\bar{j}}$ stands for the monomial
$t_1^{j_1}t_2^{j_2} \ldots t_m^{j_m}$.

By Hochster formula $$\beta_K(s,\bar{t}) = \sum\limits_{i\in
\Zo,A\subseteq [m]} \beta^{-i,2A}(K)s^{i}\bar{t}^{A},$$ where
$\bar{t}^{A} = \prod_{l\in A}t_l$. The free term, which
corresponds to $A=\varnothing, i=0$, equals $1$ for any complex
$K$. In what follows we need the reduced beta-polynomial
$$
\betar_K(s,t) = \beta_K(s,t)-1 = \sum\limits_{\substack{i\in
\Zo\\A\subseteq [m],A\neq\varnothing}}
\beta^{-i,2A}s^{i}\bar{t}^{A}.
$$
Two-parametric beta-polynomial (see \cite[sect.8]{AB}) is defined
as
$$
\bet_K(s,t) = \sum\limits_{i,j\in\Zo} \beta^{-i,2j}s^{-i}t^{2j} =
\beta_K(s^{-1},t^2,t^2,\ldots,t^2).
$$
and
$$
\bett_K(s,t) = \sum\limits_{i,j\in\Zo, j\neq 0}
\beta^{-i,2j}s^{-i}t^{2j} =
\tilde{\beta}_K(s^{-1},t^2,t^2,\ldots,t^2) = \bet_K(s,t)-1.
$$

\ex \label{examplBetaOfPartDelta} Let $K = \partial \Delta_{[m]}$.
Then by Hochster formula we have
$$\beta_{\partial\Delta_{[m]}}(s,\ts) = 1+st_1t_2\ldots t_m$$ since
the nonacyclic full subcomplexes of $K$ are only $K_{\varnothing}$
and $K_{[m]}=K$. These subcomplexes have nontrivial reduced
cohomology in dimensions $-1$ and $m-2$ respectively.

\ex \label{examplBetaOfO} Let $K=o^m$. Then
\begin{equation*}
\beta_{o^m}(s,\ts) = \sum\limits_{A\subseteq[m]} s^{|A|}\ts^{A},
\end{equation*}
since for any $A\subseteq[m]$ the full subcomplex $(o^m)_{A}$ is
empty and its $(-1)$-cohomology has rank $1$. Therefore,
\begin{equation}\label{equatBetaOfO}
\beta_{o^m}(s,\ts) = (1+st_1)\cdot\ldots\cdot(1+st_m).
\end{equation}
\\
\\
For the polytope $P$ we define the polynomials
$\beta,\betar,\bet,\bett$ as the polynomials of the corresponding
nerve-complex $K_P$:
\begin{align}
\beta_P(s,\ts) = \beta_{K_P}(s,\ts),\qquad& \betar_P(s,\ts) =
\betar_{K_P}(s,\ts),\\
\bet_P(s,t) = \bet_{K_P}(s,t),\qquad& \bett_P(s,t) =
\bett_{K_P}(s,t).
\end{align}

Now we generalize some results of \cite{AB} concerning
beta-polynomials. Our goal is to express
$\beta_{K(\Ks)}(s,\bar{t})$ in terms of $\beta_K(s,\bar{t})$ and
$\beta_{K_i}(s,\bar{t})$. To do this at first we investigate the
structure of full subcomplexes in $K(\Ks)$.

\begin{lemma}\label{lemmaIterFullSub}
Consider $K$ on $m$ vertices and $K_\alpha$ on $l_\alpha$ vertices
for $\alpha\in[m]$, so the set of vertices of $K(\Ks)$ is
$[l_1]\sqcup\ldots\sqcup[l_m]$. Let $A$ be the subset of
$[l_1]\sqcup\ldots\sqcup[l_m]$, $A = A_1\sqcup\ldots\sqcup A_m$
where $A_\alpha\subseteq[l_\alpha]$. Let
$\nu=\{\alpha_1,\ldots,\alpha_k\} = \{\alpha\in[m]\mid
A_\alpha\neq\varnothing\}$. Then
$$
K(K_1,\ldots,K_m)_{A} =
K|_{\nu}(K_{\alpha_1}|_{A_{\alpha_1}},K_{\alpha_2}|_{A_{\alpha_2}},\ldots,K_{\alpha_k}|_{A_{\alpha_k}}).
$$
\end{lemma}

The proof follows from definitions.

\begin{thm}\label{thmBettiOfComposit}
Let $K$ be the complex on $m$ vertices and $K_1,\ldots,K_m$ be
simplicial complexes on $l_1,\ldots,l_m$ vertices. Let
$\ts_j=(t_{j1},\ldots,t_{jl_j})$ for $j\in[m]$ and $$\ts =
(t_{11},\ldots,t_{1l_1},\ldots,t_{m1},\ldots,t_{ml_m}) =
(\ts_1,\ldots,\ts_m).$$ Then
\begin{equation}\label{equatBetaOfCompMain}
\beta_{K(K_1,\ldots,K_m)}(s,\ts) =
\beta_{K}(s,s^{-1}\betar_{K_1}(s,\ts_1),s^{-1}\betar_{K_2}(s,\ts_2),\ldots,s^{-1}\betar_{K_m}(s,\ts_m)).
\end{equation}
\end{thm}

\begin{proof}
Using Hochster formula \ref{equatHochster} we may write

\begin{multline}\label{equatBetti1}
\beta_{K(\Ks)}(s,\ts)=\sum\limits_{A\subseteq[l_1]\sqcup\ldots\sqcup[l_m]}
\sum\limits_{i'}\dim\Hr^{|A|-i'-1}(K(\Ks)_{A};\ko)s^{i'}\ts^{A}
=\\
\sum\limits_{A\subseteq[l_1]\sqcup\ldots\sqcup[l_m]}
\sum\limits_{i'}H_{i',A}s^{i'}\ts^{A},
\end{multline}
where $H_{i',A}$ denote the dimensions of cohomology groups. Any
subset $A\in[l_1]\sqcup\ldots\sqcup[l_m]$ is given by $A =
A_1\sqcup\ldots\sqcup A_k$ for some
$B=\{\alpha_1,\ldots,\alpha_k\}\subseteq[m]$ and $A_{1}\subseteq
[l_{\alpha_1}], \ldots, A_{k}\subseteq [l_{\alpha_k}]$ subject to
the condition $A_i\neq\varnothing$. The sum in \eqref{equatBetti1}
can be expanded

\begin{multline}\label{equatBetti2}
\sum\limits_{A\subseteq[l_1]\sqcup\ldots\sqcup[l_m]}
\sum\limits_{i'}H_{i',A}s^{i'}\ts^{A}
=\\
\sum\limits_{i'}\sum\limits_{B=\{\alpha_1,\ldots,\alpha_k\}\subseteq[m]}
\left(\sum\limits_{\substack{A_1\subseteq[l_{\alpha_1}]\\A_1\neq\varnothing}}
\ldots
\sum\limits_{\substack{A_k\subseteq[l_{\alpha_k}]\\A_1\neq\varnothing}}
H_{i',A} s^{i'}\ts_{\alpha_1}^{A_1}\ldots
\ts_{\alpha_k}^{A_k}\right).
\end{multline}

Quantities $H_{i',A}$ can be expressed using lemma
\ref{lemmaIterFullSub} and corollary \ref{corComposIsJoin}:

\begin{multline}\label{equatBetti3}
H_{i',A} = \dim\Hr^{|A|-i'-1}(K(\Ks)_{A};\ko)
=\\
\dim \Hr^{|A_1|+\ldots+|A_k|-i'-1}
(K|_B(K_{\alpha_1}|_{A_1},\ldots,K_{\alpha_k}|_{A_k});\ko)
=\\
\dim \Hr^{|A_1|+\ldots+|A_k|-i'-1} (K|_B\ast
K_{\alpha_1}|_{A_1}\ast\ldots\ast K_{\alpha_k}|_{A_k};\ko).
\end{multline}

The cohomology group of the join can be expanded

\begin{multline}\label{equatBetti4}
\dim \Hr^{|A_1|+\ldots+|A_k|-i'-1} (K|_B\ast
K_{\alpha_1}|_{A_1}\ast\ldots\ast K_{\alpha_k}|_{A_k};\ko)
= \\
\sum\limits_{\substack{r,r_1,\ldots,r_k\\r+r_1+\ldots+r_k=|A_1|+\ldots+|A_k|-i'-1-k}}
\dim\Hr^{r}(K|_B;\ko)\cdot\dim\Hr^{r_1}(K_{\alpha_1}|_{A_1};\ko)\cdot
\ldots\cdot \dim\Hr^{r_k}(K_{\alpha_k}|_{A_k};\ko).
\end{multline}

Consider indices $i,i_1,\ldots,i_k$ satisfying the identities
$r=k-i-1 = |B|-i-1$, $r_s=|A_s|-i_s-1$ for $s\in [k]$. Since
$r+\sum_{s\in[k]} r_s = \left(\sum_{s\in[k]} |A_s|\right)
-i'-1-k$, we get $i'=i-k+\sum_{s\in [k]} i_s$. Then

\begin{multline}\label{equatBetti5}
\sum\limits_{r,r_1,\ldots,r_k}
\dim\Hr^{r}(K|_B;\ko)\cdot\left(\prod\limits_{j=1}^k\dim\Hr^{r_j}(K_{\alpha_j}|_{A_j};\ko)\right)
s^{i'}\ts_{\alpha_1}^{A_1}\ldots \ts_{\alpha_k}^{A_k}
=\\
\sum\limits_{i_1,\ldots,i_k}\dim\Hr^{k-i-1}(K|_B;\ko)s^i
\prod\limits_{j=1}^k
\left(s^{-1}\dim\Hr^{|A_j|-i_j-1}(K_{\alpha_j}|_{A_j};\ko)
s^{i_j}\ts_{\alpha_j}^{A_j}\right).
\end{multline}

Therefore,

\begin{multline}\label{equatBetti6}
\sum\limits_{\substack{A_1\subseteq[l_{\alpha_1}]\\A_1\neq\varnothing}}
\ldots
\sum\limits_{\substack{A_k\subseteq[l_{\alpha_k}]\\A_k\neq\varnothing}}
\sum\limits_{i_1,\ldots,i_k}\prod\limits_{j=1}^k\left(s^{-1}\dim\Hr^{|A_j|-i_j-1}
(K_{\alpha_j}|_{A_j};\ko)s^{i_j}\ts_{\alpha_j}^{A_j}\right)
=\\
\prod\limits_{j=1}^k\left(\sum\limits_{\substack{A_j\subseteq[l_{\alpha_j}]\\A_j\neq\varnothing}}
\sum\limits_{i_j}s^{-1}\dim\Hr^{|A_j|-i_j-1}
(K_{\alpha_j}|_{A_j};\ko)s^{i_j}\ts_{\alpha_j}^{A_j}\right)
 = \\
\prod\limits_{j=1}^k\left(\sum\limits_{\substack{A_j\subseteq[l_{\alpha_j}]\\A_j\neq\varnothing}}
\sum\limits_{i_j}s^{-1}\beta^{-i_j,2A_j}(K_{\alpha_j})
s^{i_j}\ts_{\alpha_j}^{A_j}\right)
 = \prod\limits_{j=1}^k\left(s^{-1}\betar_{K_{\alpha_j}}(s,\ts_{\alpha_j})
\right).
\end{multline}

Substituting \eqref{equatBetti6} into \eqref{equatBetti2} we get

\begin{multline}
\beta_{K(\Ks)}(s,\ts)
=\\
\sum\limits_{i}\sum\limits_{B=\{\alpha_1,\cdots,\alpha_k\}\subseteq[m]}
\beta^{-i,2B}(K|_B)s^i
\prod\limits_{j=1}^k\left(s^{-1}\betar_{K_{\alpha_j}}(s,\ts_{\alpha_j})\right)
=\\
\beta_K(s,s^{-1}\betar_{K_1}(s,\ts_1),\ldots,s^{-1}\betar_{K_m}(s,\ts_m)).
\end{multline}
This finishes the proof.
\end{proof}

\begin{cor}\label{corBeta2ofKsGeneral}
$$\bet_{K(\Ks)}(s,t) = \beta(s^{-1},s\bett_{K_1}(s,t),\ldots,s\bett_{K_m}(s,t)).$$
\end{cor}

\begin{proof}
Substitute $s^{-1}$ and $t^2$ for $s$ and $t_{ji_j}$ in
\eqref{equatBetaOfCompMain} and use the definition of a
two-parametric beta-polynomial.
\end{proof}

\begin{cor}
Let $P_1$ and $P_2$ be two convex polytopes and $P_1\ast P_2$ ---
their join. Let $\ts_i=(t_{i1},\ldots,t_{il_i})$ be formal
variables corresponding to facets of $P_i$ for $i=1,2$ and
$\ts=(\ts_1,\ts_2)$. Then $$\beta_{P_1\ast P_2}(s,\ts) =
1+s^{-1}\betar_{P_1}(s,\ts_1)\betar_{P_2}(s,\ts_2)$$
$$
\bet_{P_1\ast P_2}(s,t) = 1+s\bett_{P_1}(s,t)\bett_{P_2}(s,t)
$$
\end{cor}

\begin{proof}
By example \ref{examplDeltaPolyIsJoin} we have $P_1\ast P_2 =
\triangle_{[2]}(P_1,P_2)$. By proposition \ref{propFunctorOperad}
$K_{\triangle_{[2]}(P_1,P_2)} =
K_{\triangle_{[2]}}(K_{P_1},K_{P_2}) =
\partial\Delta_{[2]}(K_{P_1},K_{P_2})$. Then by definition
\begin{multline}
\beta_{P_1\ast P_2}(s,\ts) =
\beta_{\triangle_{[2]}(P_1,P_2)}(s,\ts) =
\beta_{\partial\Delta_{[2]}}(s,s^{-1}\betar_{P_1}(s,\ts_1),s^{-1},\betar_{P_2}(s,\ts_2))
=\\
1+s^{-1}\betar_{P_1}(s,\ts_1)\betar_{P_2}(s,\ts_2).\end{multline}
Substituting $s^{-1}$ for $s$ and $t^2$ for each $t_{rj}$ gives
the second expression of the corollary. See \cite{AB} for an
independent proof of this statement.
\end{proof}

\begin{cor}\label{corBeta2ofKls}
Let $K$ be a simplicial complex on $m$ vertices and
$(l_1,\ldots,l_m)$ --- an array of nonnegative integers. Then
$$
\bet_{K(l_1,\ldots,l_m)}(s,t) =
\beta_K(s^{-1},t^{2l_1},t^{2l_2},\ldots,t^{2l_m})
$$
In particular, if $l_1=l_2=\ldots=l_m=l$ we have
$$
\bet_{K(\underline{l})}(s,t) = \bet_K(s,t^l).
$$
\end{cor}

\begin{proof}
By definition $K(l_1,\ldots,l_m) =
K(\partial\Delta_{[l_1]},\ldots,\partial\Delta_{[l_m]})$ and
$\betar_{\partial\Delta_{[l_r]}}(s,\ts_r) = st_{r1}\ldots
t_{rl_r}$ (example \ref{examplBetaOfPartDelta}). Then
\begin{multline*}
\beta_{K(l_1,\ldots,l_m)}(s,\ts) = \beta_K(s,s^{-1}st_{11}\ldots
t_{1l_1},\ldots,s^{-1}st_{m1}\ldots t_{ml_m}) =\\
\beta_K(s,t_{11}\cdot\ldots\cdot
t_{1l_1},\ldots,t_{m1}\cdot\ldots\cdot t_{ml_m}).
\end{multline*}
Substituting $s^{-1}$ for $s$ and $t^2$ for $t_{rj}$ gives the
required formula.
\end{proof}

\ex Consider the case $o^m(K_1,\ldots,K_m) = K_1\ast\ldots\ast
K_m$. Using theorem \ref{thmBettiOfComposit} and relation
\eqref{equatBetaOfO} we get
\begin{multline}
\beta_{K_1\ast\ldots\ast K_m}(s,\ts) = \beta_{K_1\ast\ldots\ast
K_m}(s,\ts) =\\
(1+s\cdot s^{-1}\betar_{K_1}(s,\ts_1))\cdot\ldots\cdot(1+s\cdot
s^{-1}\betar_{K_m}(s,\ts_m)) =
\beta_{K_1}(s,\ts_1)\cdot\ldots\cdot\beta_{K_m}(s,\ts_m).
\end{multline}
This result can be proved directly by the isomorphism
$$
\ko[K_1\ast\ldots\ast K_m] \cong
\ko[K_1]\otimes\ldots\otimes\ko[K_m].
$$
and the definition of multigraded Betti numbers.

\section{Enumerative polynomials}\label{sectEnumPolynomials}

Let $K$ be a simplicial complex. For each $i\geqslant 0$ define a
number $f_{i} = |\{I\in K\mid|I|=i\}|$. The polynomial
$$
f_K(t) = \sum\limits_if_it^i = \sum\limits_{I\in K}t^{|I|}
$$
is called an $f$-polynomial of $K$. If $\dim K = n-1$, then $\deg
f_K(t) = n$. The $h$-numbers $h_i$ are defined by the relation
$$
h_0t^n+\ldots+h_{n-1}t+h_n =
f_0(t-1)^n+f_1(t-1)^{n-1}+\ldots+f_{n}.
$$
The polynomial $h_K(t)=h_0+h_1t+\ldots+h_nt^n$ is called the
$h$-polynomial of the complex $K$. Writing the defining relations
for $h_i$ we have
\begin{equation}\label{equatFtoHpoly}
h_K(t) = (1-t)^nf_K\left(\dfrac{t}{1-t}\right).
\end{equation}
Since the relation \eqref{equatFtoHpoly} is invertible, $h$-and
$f$-polynomials carry the same combinatorial information. The
$h$-polynomial is connected to Hilbert--Poincare series of the
algebra $\ko[K]$ with $\Zo$-grading by the formula
\cite{St2},\cite{BP}:
\begin{equation}\label{equatHilbStanReisRing}
\Hilb(\ko[K];t) = \frac{h_K(t^2)}{(1-t^2)^n}.
\end{equation}

There is a formula which connects $h$-polynomial of $K$ with
bigraded Betti numbers. Let $\chi_j(K) =
\sum_{i=0}^m(-1)^i\beta^{-i,2j}(K)$ and
$\chi_K(t)=\sum_{j=0}^m\chi_j(K)t^{2j}$. Then by \cite[Theorem
7.15]{BP}
\begin{equation}\label{equatHtoChiConnection}
\chi_K(t) = (1-t^2)^{m-n}h_K(t^2) = (1-t^2)^m\Hilb(\ko[K];t).
\end{equation}

Since $\chi_K(t) = \bet_K(-1,t)$ we get a simple formula

\begin{equation}\label{equatHtoBetaConnection}
\bet(-1,t) = (1-t^2)^{m-n}h_K(t^2).
\end{equation}

Equation \ref{equatHtoBetaConnection} allows to express the
$h$-polynomial of $lK = K(l,\ldots,l)$ in terms of the
$h$-polynomial of $K$.

\begin{prop}\label{propHofKlll}
Let $K$ be $(n-1)$-dimensional complex on $m$ vertices, $l>0$ and
$lK = K(l,l,\ldots,l) =
K(\partial\Delta_{[l]},\ldots,\partial\Delta_{[l]})$. Then
$$
h_{lK}(t) = (1+t+\ldots+t^{l-1})^{m-n}h_K(t^l)
$$
\end{prop}

\begin{proof}
The complex $lK$ has $m'=ml$ vertices. It can be seen that
$n'=\dim lK+1 = nl+(m-n)(l-1)$. Then $m'-n' = m-n$. By relation
\eqref{equatHtoBetaConnection} $\bet_{lK}(-1,t) =
(1-t^2)^{m'-n'}h_{lK}(t^2)$. On the other hand, by corollary
\ref{corBeta2ofKls} $\bet_{lK}(s,t) = \bet_K(s,t^l)$, therefore
$\bet_{lK}(-1,t)=\bet_K(-1,t^l)$. This gives a sequence of
equalities:
$$
(1-t^2)^{m-n}h_{lK}(t^2) = (1-t^2)^{m'-n'}h_{lK}(t^2) =
\bet_{lK}(-1,t) = \bet_K(-1,t^l) = (1-t^{2l})^{m-n}h_K(t^{2l}).
$$

Therefore, $h_{lK}(t^2) =
\left(\frac{1-t^{2l}}{1-t^2}\right)^{m-n}h_K(t^{2l}) =
(1+t^2+\ldots+t^{2(l-1)})^{m-n}h_K(t^{2l})$.
\end{proof}

In particular for $l=2$ this gives
$h_{2K}(t)=(1+t)^{m-n}h_K(t^2)$. This result is proved in
\cite{Ust} by another method.

\rem The result of proposition \ref{propHofKlll} can be proved
independently using formula \eqref{equatHilbStanReisRing} by
studying the structure of the ring $\ko[lK]$ (see \cite{BBCGit}
for details).
\\
\\
It is convenient to introduce another polynomial $q_K(t)$ while
working with the composition of simplicial complexes. For an
$(n-1)$-dimensional complex $K$ with $m$ vertices let $$q_K(t) =
1-(1-t)^{m-n}h_K(t).$$

\ex It is known that $h_{\partial\Delta_{[m]}}(t) =
1+t+\ldots+t^{m-1}$. Then $q_{\partial\Delta_{[m]}}(t) = t^m$.
\\
\\
We have a formula
\begin{equation}\label{equatQtoBeta}
\bett_K(-1,t) = \bet_K(-1,t)-1 = (1-t^2)^{m-n}h_K(t)-1 =
-q_K(t^2).
\end{equation}

Also we have
\begin{equation}\label{equatQtoMultuBeta}
\betar_{K}(-1,t,\ldots,t)=-q_K(t)
\end{equation}

\begin{prop}
Consider arbitrary simplicial complexes $K_1,\ldots,K_m$. Then
$$
q_{\partial\Delta_{[m]}(K_1,\ldots,K_m)} =
q_{K_1}(t)\cdot\ldots\cdot q_{K_m}(t)
$$
\end{prop}

\begin{proof}
By corollary \ref{corBeta2ofKsGeneral}
$\bett_{\partial\Delta_{[m]}(\Ks)}(s,t) =
\betar_{\partial\Delta_{[m]}}(s^{-1},s\bett_{K_1}(s,t),\ldots,s\bett_{K_m}(s,t))
=
s^{-1}\cdot(s\bett_{K_1}(s,t))\cdot\ldots\cdot(s\bett_{K_m}(s,t))$.
Substituting $s=-1$ and using formula \eqref{equatQtoBeta} gives
the required relation.
\end{proof}

\begin{prop}
For any nonempty complexes $K$ and $L$ there holds
$$
q_{K(L,\ldots,L)}(t) = q_{K}(q_L(t)).
$$
\end{prop}

\begin{proof}
By corollary \ref{corBeta2ofKsGeneral} $\bett_{K(L,\ldots,L)}(s,t)
= \betar_K(s^{-1}, s\bett_L(s,t),\ldots, s\bett_L(s,t))$.
Substituting $s=-1$ gives
$$
-q_{K(L,\ldots,L)}(t^2) = \betar_K(-1,q_L(t^2),\ldots,q_L(t^2)) =
-q_K(q_L(t^2))
$$
which was to be proved.
\end{proof}


\begin{thebibliography}{99}

\bibitem{Agn} Geir Agnarsson \textit{The flag polynomial of the Minkowski sum of
simplices}, arXiv:1006.5928

\bibitem{An} D. Anick \textit{Connections between Yoneda and Pontrjagin algebras},
Algebraic topology, Aarhus 1982, 331--350, Lecture Notes in Math.,
1051, Springer, Berlin, 1984.

\bibitem{AB} A. A. Ayzenberg, V. M. Buchstaber, \textit{Moment-angle spaces and
nerve-complexes of convex polytopes}, Proceedings of the Steklov
Institute of Mathematics, V.275, 2011.

\bibitem{BBCG} A. Bahri, M. Bendersky, F. R. Cohen, S. Gitler, \textit{The
polyhedral product functor: A method of decomposition for
moment-angle complexes, arrangements and related spaces}, Advances
in Mathematics, 225:3 (2010), 1634--1668.

\bibitem{BBCGit} A. Bahri, M. Bendersky, F. R.
Cohen, S. Gitler, \textit{A new topological construction of
infinite families of toric manifolds implying fan reduction},
arXiv:1011.0094v3

\bibitem{Bask} I.\,V.\,Baskakov \textit{Cohomology of K-powers of spaces and the combinatorics of
simplicial divisions}, Russian Mathematical Surveys
(2002),57(5):989.


\bibitem{BP} V.\,M.\,Buchstaber and T.\,E.\,Panov, Torus Actions and Their
Applications in Topology and Combinatorics // University Lecture,
vol. 24, Amer. Math. Soc., Providence, R.I., 2002.

\bibitem{BP2} V.\,M.\,Bukhshtaber, T.\,E.\,Panov, Torus actions, combinatorial
topology, and homological algebra // Russian Math. Surveys 55
(2000), Number 5, 825--921.

\bibitem{BPnew} V.\,M.\,Buchstaber and T.\,E.\,Panov, Toric
Topology // arXiv:1210.2368

\bibitem{Hir}
Philip S. Hirschhorn \textit{Model Categories and Their
Localizations}. Volume 99 of Mathematical Surveys and Monographs,
AMS, Providence, RI, 2003.

\bibitem{Hoch} M.\,Hochster, \textit{Cohen-Macaulay rings, combinatorics, and simplicial
complexes}, in Ring theory, II (Proc. Second Conf.,Univ. Oklahoma,
Norman, Okla., 1975),  Lecture Notes in Pure and Appl. Math., vol.
26, 171--223, Dekker, New York, 1977.

\bibitem{ML}
S. Maclane, \textit{Categories for the working mathematician},
Graduate Texts in Mathematics 5 (2nd ed.). Springer-Verlag,
(1998).

\bibitem{PR}
Taras Panov, Nigel Ray, \textit{Categorical aspects of toric
topology}, in "Toric Topology" (M.Harada et al, eds.),
Contemporary Mathematics, vol.460, American Mathematical Society,
Providence, RI, 2000, pp.293--322.

\bibitem{PB} J. S. Provan and L. J. Billera, \textit{Decompositions of simplicial
complexes related to diameters of convex polyhedra}, Mathematics
of Operations Research, volume 5, (1980), 576–-594.

\bibitem{St2}
R.\,Stanley, \textit{Combinatorics and Commutative Algebra},
Boston, MA: Birkh\"{a}user Boston Inc., 1996. (Progress in
Mathematics V. 41).

\bibitem{Ust}
Yury Ustinovsky \textit{Doubling operation for polytopes and torus
actions}, Russian Math. Surveys 64 (2009) no.5.

\bibitem{UstTR}
Yury Ustinovsky \textit{Toral rank conjecture for moment-angle
complexes}, arXiv:0909.1053v2

\bibitem{WZZ}
Volkmar Welker, G\"{u}nter M.\,Ziegler, Rade
T.\,\v{Z}ivaljevi\'{c}, \textit{Homotopy colimits --- comparison
lemmas for combinatorial applications}, Journal fur die reine und
angewandte Mathematik (Crelles Journal). Volume 1999, Issue 509,
Pages 117--149, ISSN (Online) 1435--5345, ISSN (Print) 0075-4102,
DOI: 10.1515/crll.1999.509.117, April 1999.

\bibitem{Zieg} G\"{u}nter M.\,Ziegler, \textit{Lectures on Polytopes},
Springer-Verlag, New York, 2007.

\end{thebibliography}
\end{document}